\documentclass[letterpaper, 11pt,  reqno]{amsart}

\usepackage{amsmath,amssymb,amscd,amsthm,amsxtra, esint}

\usepackage{hyperref} 
\allowdisplaybreaks[2]

\newtheorem{theorem}{Theorem} [section]

\newtheorem{lemma}[theorem]{Lemma}
\newtheorem{proposition}[theorem]{Proposition}
\newtheorem{remark}[theorem]{Remark}

\DeclareMathOperator*{\supp}{supp}

\newcommand{\noi}{\noindent}
\newcommand{\Z}{\mathbb{Z}}
\newcommand{\R}{\mathbb{R}}

\newcommand{\T}{\mathbb{T}}

\let\Re=\undefined\DeclareMathOperator*{\Re}{Re}

\let\P= \undefined
\newcommand{\P}{\mathbf{P}}

\newcommand{\F}{\mathcal{F}}

\newcommand{\al}{\alpha}
\newcommand{\be}{\beta}
\newcommand{\dl}{\delta}

\newcommand{\eps}{\varepsilon}

\newcommand{\G}{\Gamma}
\newcommand{\ld}{\lambda}

\newcommand{\s}{\sigma}

\newcommand{\ft}{\widehat}

\newcommand{\wt}{\widetilde}
\newcommand{\cj}{\overline}
\newcommand{\dx}{\partial_x}
\newcommand{\dt}{\partial_t}
\newcommand{\dd}{\partial}

\renewcommand{\l}{\ell}

\newcommand{\les}{\lesssim}
\newcommand{\ges}{\gtrsim}

\newcommand{\jb}[1]
{\langle #1 \rangle}

\newcommand{\ind}{\mathbf 1}

\newcommand{\M}{\mathcal{M}}

\newcommand{\N}{\mathbb{N}}

\newcommand{\NN}{\mathcal{N}}

\newtheorem*{ackno}{Acknowledgments}

\numberwithin{equation}{section}
\numberwithin{theorem}{section}

\renewcommand{\u}{{\bf u}}

\begin{document}

\title[On the ill-posedness of the  cubic NLS on the circle]
{On the ill-posedness  of the   cubic nonlinear
Schr\"odinger equation on the circle}

\author
{Tadahiro Oh and Yuzhao Wang}

\address{
Tadahiro Oh, 
School of Mathematics\\
The University of Edinburgh\\
and The Maxwell Institute for the Mathematical Sciences\\
James Clerk Maxwell Building\\
The King's Buildings\\
Peter Guthrie Tait Road\\
Edinburgh\\ 
EH9 3FD\\
 United Kingdom}

\email{hiro.oh@ed.ac.uk}

\address{
Yuzhao Wang, 
School of Mathematics\\
The University of Edinburgh\\
and The Maxwell Institute for the Mathematical Sciences\\
James Clerk Maxwell Building\\
The King's Buildings\\
Peter Guthrie Tait Road\\
Edinburgh\\ 
EH9 3FD\\
 United Kingdom}

\email{yuzhao.wang@ed.ac.uk}



\begin{abstract}
In this note, we 
consider the ill-posedness issue for  the cubic nonlinear Schr\"odinger equation
(NLS)
on the circle.
In particular,  adapting the argument by 
Christ-Colliander-Tao \cite{CCT2b} to the periodic setting, 
we exhibit a norm inflation phenomenon
for both the usual cubic NLS
and the Wick ordered cubic NLS 
for $s \leq s_\text{crit} :=- \frac 12$.
We also discuss norm inflation phenomena for  general 
cubic fractional NLS on the circle.

\end{abstract}

\maketitle

\section{Introduction}

We consider 
the  cubic  nonlinear Schr\"odinger equation (NLS): 
\begin{align}
\begin{cases}
i \dt u -    \dx^2 u \pm  |u|^{2}u  = 0\\
u|_{t = 0} = u_0, 
\end{cases}
\ (x, t) \in \T\times \R \quad \text{or} \quad \R \times \R.
\label{NLS0}
\end{align}

\noi
The equation \eqref{NLS0} appears
in various physical settings: nonlinear optics, 
fluids, plasmas, 
and quantum field theory.  See \cite{SULEM} for a general review on the subject.
It is  also known to be one of the simplest completely integrable 
partial differential equations (PDEs)
\cite{AKNS, AM, GK}.

\subsection{Galilean invariance}
The Cauchy problem for the cubic NLS has been studied extensively
both on $\R$ and $\T$.
It is well known that \eqref{NLS0} is invariant under several symmetries.
The Galilean invariance
 states that if $u(x, t)$ is a solution to \eqref{NLS0}
on $\R$, then
$u^\beta(x, t) := e^{i\frac{\beta}{2}x}e^{i\frac{\beta^2}{4}t}  u (x+\beta t, t)$
is also a solution to \eqref{NLS0} on $\R$ with  modulated initial data. 
Note that the $L^2$-norm is preserved under this Galilean invariance.\footnote{In fact, 
any Fourier-Lebesgue $\F L^{p}$-norm defined in \eqref{FL} is invariant under the Galilean symmetry.}
Namely, 
$s_\text{crit}^\infty  :=  0$ is the critical Sobolev regularity
with respect to the Galilean invariance.
In fact, while \eqref{NLS0} is globally well-posed
in $L^2(\R)$ \cite{Tsutsumi},  
it is known to be `mildly ill-posed' below $L^2(\R)$.
More precisely, the solution map $\Phi(t): u_0 \in H^s(\R) \mapsto u(t) \in H^s(\R)$
fails to be locally uniformly continuous if $s < 0$ \cite{KPV, CCT1}.
We point out that this mild ill-posedness result does 
not assert that \eqref{NLS0} on $\R$ is ill-posed below $L^2(\R)$.
Indeed, the well-posedness issue of \eqref{NLS0} on $\R$ below $L^2(\R)$
is a long-standing open problem.

The Galilean invariance also holds in the periodic setting for $\be \in 2 \Z$.
In this case, there is even a stronger dichotomy in the behavior of solutions for  $s\geq 0$ and $s < 0$.
On the one hand, Bourgain \cite{BO1} 
proved global well-posedness of \eqref{NLS0} in $L^2(\T)$.
On the other hand, \eqref{NLS0} on $\T$  
is known to be ill-posed below $L^2(\T)$.
As in the real line case, 
the solution map $\Phi(t): u_0\in H^s (\T)\mapsto u(t) \in H^s(\T)$
fails to be locally uniformly continuous if $s < 0$ \cite{BGT, CCT1}.
Moreover, Christ-Colliander-Tao \cite{CCT2} and Molinet \cite{MOLI}
proved that the solution map is in fact discontinuous if $s < 0$.
Hence, \eqref{NLS0} on $\T$ is ill-posed in negative Sobolev spaces.
Finally,  Guo-Oh \cite{GO} proved non-existence of solutions for \eqref{NLS0} on $\T$ below $L^2(\T)$.
See \cite{GO} for a precise statement.

Therefore, in the periodic setting,  one needs
to consider a renormalized cubic NLS outside $L^2(\T)$.
Given a global solution $u \in C(\R;  L^2(\T))$ to \eqref{NLS0}, 
we define the following invertible gauge transformation:
\begin{equation}
u(t) \mapsto
\mathcal{G}(u)(t) : = e^{\mp 2 i t \fint |u|^2 dx} u(t).
\label{gauge}
\end{equation}

\noi
In view of the $L^2$-conservation, 
a direct computation shows that the gauged function,  which we still denote by $u$, 
solves the following
 {\it Wick ordered cubic NLS}:
\begin{equation}
	\label{NLS1} 
	\begin{cases}
		i \dt u - \dx^2 u \pm ( |u|^2 -2  \fint |u|^2 dx) u = 0 \\
		u|_{t= 0} = u_0,
	\end{cases}
	(x, t)  \in \T \times \R.
\end{equation}

\noi
Conversely, given a global solution $u \in C(\R;  L^2(\T))$
to \eqref{NLS1}, 
we see that $\mathcal{G}^{-1}(u)$ solves the original cubic NLS \eqref{NLS0}.
Such a gauge transformation, however, 
does not make sense below $L^2(\T)$
and thus
we cannot freely convert solutions of \eqref{NLS1} into solutions of \eqref{NLS0}.
In other words, 
 the standard  cubic NLS \eqref{NLS0}
and the Wick ordered  cubic NLS \eqref{NLS1} 
 represent the same dynamics  under different gauge choices
(only) in $L^2(\T)$.

The renormalization of the nonlinearity in \eqref{NLS1}
is canonical, corresponding to the Wick ordering 
in Euclidean quantum field theory.
See  \cite{BO96, OS, OThomann}.
The specific choice of  a gauge for \eqref{NLS1}
 removes a certain singular component from the cubic nonlinearity.
 As a result, 
the Wick ordered cubic NLS \eqref{NLS1}  
is known to behave better than  the cubic NLS \eqref{NLS0} outside $L^2(\T)$.
In fact, the standard cubic NLS on $\R$
and the Wick ordered cubic NLS \eqref{NLS1} on $\T$ share many common features:

\begin{itemize}
\item global well-posedness in $L^2$ \cite{Tsutsumi, BO1}. 

\vspace{1mm}

\item failure of uniform continuity of the solution map on bounded sets  below $L^2$
\cite{KPV, BGT, CCT1,  CO}.

\vspace{1mm}

\item weak continuity in $L^2$ \cite{GM, OS}. 

\vspace{1mm}

\item local well-posedness in the Fourier-Lebesgue spaces $\F L^{s, p}$ \cite{Grunrock, CH2, GH}
for $ s\geq 0$ and $1 < p < \infty$.
Here, the Fourier-Lebesgue space $\F L^{s, p}(\R)$ is defined by the norm:
\begin{align}
\|f \|_{\F L^{s, p}(\R)} = \| \jb{\xi}^s \ft f (\xi)\|_{L^{p}(\R)}
\label{FL}
\end{align}

\noi
with an analogous definition for $\F L^{s, p}(\T)$ in the periodic setting.
When $s = 0$, we set $\F L^{p} := \F L^{0, p}$ for simplicity.

\vspace{1mm}

\item a priori bound and existence (without uniqueness) of solutions in negative Sobolev spaces
\cite{KT1, CCT3, KT2, GO}.

\end{itemize}

\noi
See \cite{OS} for more discussion on this issue.

\subsection{Scaling  symmetry}

Another important symmetry is the dilation symmetry; 
 if $u(x, t)$ is a solution to \eqref{NLS0}
on $\R$, then
$u^\ld(x, t) := \ld^{-1} u (\ld^{-1}x, \ld^{-2}t)$
is also a solution to \eqref{NLS0} on $\R$ with scaled initial data.
This dilation symmetry
induces another critical Sobolev regularity  given by $s_\text{crit} := -\frac 12$.
Namely, the dilation symmetry leaves the 
 homogeneous $\dot{H}^{s_\text{crit}}$-norm invariant.
 It is commonly conjectured that a PDE is ill-posed in $H^s$ for $s < s_\text{crit}$.
Indeed, Christ-Colliander-Tao \cite{CCT2b} exhibited
a {\it norm inflation} phenomenon for \eqref{NLS0} in  $H^s(\R)$
for $s <  s_\text{crit} = -\frac{1}{2}$.\footnote{
Although Theorem 2 in \cite{CCT2b} claims norm inflation for $s \leq - \frac {d}{2}$, 
the proof requires $s < s_\text{crit}$ and thus, as it is written,  their result does not hold for the cubic NLS \eqref{NLS0} on $\R$
at the scaling critical regularity $s = - \frac 12$.
  See (4.4) in \cite{CCT2b}.
One can, however, modify the argument and show that  norm inflation holds even 
at the scaling critical regularity.  See Subsection \ref{SUBSEC:crit} below.
See also \cite{Kishimoto, Oh, CP}.
}
Note that this is a stronger statement than the failure of continuity of the solution map (at 0).
On the one hand, the failure of continuity at 0
states that 
given any $\eps > 0$, 
there exist a solution $u = u(\eps)$ to \eqref{NLS0}
and a time $t  \in (0, \eps) $ such that 
\[ \| u(0)\|_{H^s(\R)} < \eps \qquad \text{ and } \qquad \| u(t)\|_{H^s(\R)} \ges 1.\]

\noi
On the other hand,  norm inflation 
states that 
given any $\eps > 0$, 
there exist a solution $u = u(\eps) $ to \eqref{NLS0} on $\R$
and a time $t  \in (0, \eps) $ such that 
\[ \| u(0)\|_{H^s(\R)} < \eps \qquad \text{ and } \qquad \| u(t)\|_{H^s(\R)} > \eps^{-1}.\]

While there is no dilation symmetry in the periodic setting, 
the heuristics provided by the scaling argument still plays an important role.
In the following, we prove that the same norm inflation phenomenon
also holds for both the standard cubic NLS \eqref{NLS0} and the Wick ordered NLS \eqref{NLS1} 
in the periodic setting if $s \leq -\frac 12$.\footnote{In a recent paper \cite{Oh}, 
the first author showed that the norm inflation for \eqref{NLS0} on $\M = \R$ or $\T$
also holds for general initial data in the following sense;
let  $s \leq  - \frac 12$. 
Then, 
given any $u_0 \in H^s(\M)$ and  any $\eps > 0$, 
there exist a  solution $u = u(\eps, u_0)$ to \eqref{NLS0} on $\M$
and a time $t  \in (0, \eps) $ such that 
\begin{align*}
 \| u(0) - u_0 \|_{H^s(\M)} < \eps \qquad \text{ and } 
\qquad \| u_\eps(t)\|_{H^s(\M)} > \eps^{-1}.
\end{align*}

\noi
This in particular implies that 
the solution map 
to the cubic NLS \eqref{NLS0} is discontinuous
everywhere in $H^s(\M)$.
A similar norm inflation with general initial data
holds for the Wick ordered NLS \eqref{NLS1} on $\T$. 
}

\begin{theorem}\label{THM:1}
Let $s \leq - \frac 12$.
Then, 
given any $\eps > 0$, 
there exist a solution $u = u(\eps)$ to \eqref{NLS0} on $\T$
and a time $t  \in (0, \eps) $ such that 
\[ \| u(0)\|_{ H^s(\T)} < \eps \qquad \text{ and } \qquad \| u(t)\|_{ H^s(\T)} > \eps^{-1}.\] 

\end{theorem}

While we state and prove Theorem \ref{THM:1} 
for the standard cubic NLS \eqref{NLS0} on $\T$, 
it is easy to see that the same result holds for 
the Wick ordered NLS \eqref{NLS1} on $\T$
since the gauge transformation \eqref{gauge}
preserves any Sobolev norm (thanks to the $L^2$-conservation).

We prove Theorem \ref{THM:1} by adapting
the argument in \cite{CCT2b} to the periodic setting.
In particular, we exploit 
the high-to-low energy transfer  in the dynamics of 
the following
(dispersionless) ODE: 
\begin{align}
i \dt w \pm |w|^2 w  = 0
\label{NLS1a}
\end{align}

\noi
and  the small dispersion NLS:
\begin{align*}
i \dt v -   \dl^2  \dx^2 v \pm  |v|^{2}v  = 0, 
\end{align*}

\noi
approximating \eqref{NLS1a}.
In \cite{CCT2b}, scaling  plays an important role.
In the periodic setting, however, 
scaling alters spatial domains
and thus we must work with care.
We also use the periodization by Dirac's comb
to transfer some information from $\R$
to dilated tori of large periods.
We point out that in the critical case: $ s= -\frac 12$, 
we employ solutions to \eqref{NLS1a}
exhibiting  stronger high-to-low energy transfer than 
those in \cite{CCT2b}.
Moreover, we do not use a scaling argument in this  case.
Instead, we directly approximate NLS \eqref{NLS0} by the ODE \eqref{NLS1a}.
See Subsection \ref{SUBSEC:crit}
and Appendix \ref{SEC:APP2}.
While this argument can be  adapted to 
the higher dimensional  setting (on both rational and irrational tori), 
we do not pursue this issue here.

In \cite{IO}, Iwabuchi-Ogawa introduced 
a technique for proving ill-posedness of 
evolution equations, exploiting high-to-low energy transfer
in the first Picard iterate.
This method is built upon the previous work by Bejenaru-Tao \cite{BT}
and is developed further to cope with a  wider class of equations, 
utilizing (scaled) modulation spaces.
In fact, Theorem \ref{THM:1} can also be proven
by this method.
See Kishimoto \cite{Kishimoto} for details.
See also a recent result by Carles-Kappeler \cite{CK}, 
where they employed  a geometric optics approach 
and proved a norm inflation for \eqref{NLS0} and \eqref{NLS1} 
in the Fourier-Lebesgue space $\F L^{s, p}(\T)$
for  $ s< -\frac 23$ and $p \in [1, \infty]$.
While our method in this paper can be  adapted
to the Fourier-Lebesgue spaces, 
we do not pursue this issue here.

Lastly, we point out 
the work by Burq-G\'erard-Tzvetkov \cite[Appendix]{BGT2}
on ill-posedness in $H^1(\mathcal M)$
of super-quintic NLS on a
three-dimensional Riemannian manifold $\mathcal{M}$.
While this work is strongly motivated by \cite{CCT2b}
utilizing the dispersionless NLS,  
they
exploit the local-in-space nature of the dynamics
responsible for norm inflation
in the absence of dilation symmetry.
This would provide an alternative approach to Theorem \ref{THM:1}.

%

\subsection{Norm inflation for the cubic fractional NLS}

In the following, we briefly discuss the situation for
the  cubic  fractional NLS on the circle:
\begin{align}
\begin{cases}
i \dt u +  (-  \dx^2)^\al u \pm  |u|^{2}u  = 0 \\
u|_{t = 0} = u_0, 
\end{cases}
\ (x, t) \in \T\times \R, 
\label{HNLS1}
\end{align}

\noi
for $\al > 0$.
When $\al = \frac 12$, \eqref{HNLS1} corresponds
to the cubic half-wave equation, while 
it corresponds to the cubic fourth order NLS when $\al = 2$.
See, for example, \cite{GG, OT, OW} for the 
work on \eqref{HNLS1}
in these cases.

The dilation symmetry for \eqref{HNLS1} on $\R$
states that  if $u(x, t)$ is a solution to \eqref{HNLS1}
on $\R$, then
$u^\ld(x, t) := \ld^{-\al} u (\ld^{-1}x, \ld^{-2\al}t)$
is also a solution to \eqref{HNLS1} on $\R$ with scaled initial data.
This gives rise to the scaling critical Sobolev regularity
$s_\text{crit}^\al: = \frac{1}{2} - \al$.

We prove the following norm inflation phenomenon for \eqref{HNLS1}.

\begin{theorem}\label{THM:2}
Given $\al > 0$, 
let \textup{(i)}
$s \leq  s_\textup{crit}^\al$ if
$s_\textup{crit}^\al \leq -\frac 12 $
or 
\textup{(ii)}
$s <   s_\textup{crit}^\al$ 
and $s\ne 0$ 
if $s_\textup{crit}^\al >-\frac 12  $.
Then, 
given any $\eps > 0$, 
there exist a solution $u = u(\eps)$ to \eqref{HNLS1} on $\T$
and a time $t  \in (0, \eps) $ such that 
\[ \| u(0)\|_{H^s(\T)} < \eps \qquad \text{ and } \qquad \| u(t)\|_{ H^s(\T)} > \eps^{-1}.\] 

\end{theorem}

Note that, even if $s_\text{crit}^\al \geq 0$, 
 there is no norm inflation for $s = 0$ due to the $L^2$-conservation.
When $s \leq -\frac 12$, 
Theorem \ref{THM:1} follows from 
 a small modification of the proof of Theorem \ref{THM:1}. 
When $s > -\frac 12$, we need to use different energy transfer mechanisms, 
depending on $- \frac 1 2  < s < 0$ or $s > 0$.
While we state and prove Theorem \ref{THM:2} in the periodic setting, 
the proof of Theorem \ref{THM:2} can be easily adapted
to the non-periodic setting (and it is in fact  easier).

In a recent paper \cite{CP}, Choffrut-Pocovnicu 
applied  the argument in   \cite{IO, Kishimoto}
and  
proved  norm inflation for the cubic 
half-wave equation on $\T$ and $\R$ (\eqref{HNLS1} with $\al = \frac 12 $)
for $s < s_\text{crit}^\al = 0$, exploiting abundance of resonances in the half-wave equation.
They also obtained a partial result 
for general $\al > 0$.

\begin{remark}\label{REM:Yuzhao}\rm
When  $ 0 < \al <1$, 
 we have  $s_\text{crit}^\al > -\frac 12$.
In this case, the question of norm inflation
at the critical regularity $ s= s_\text{crit}^\al > -\frac 12$ with $s \ne 0$
remains open.
See Remarks \ref{REM:fail1} and \ref{REM:fail2}
for a brief discussion for $-\frac 12 < s_\text{crit}^\al <0 $.
\end{remark}

The defocusing/focusing nature of 
the equation does not play any role.
Hence, we assume that it is defocusing
(with the $+$\,sign in \eqref{NLS0} and \eqref{HNLS1}) in the following.
Moreover, in view of the time reversibility of the equations, 
we only consider positive times.

\section{Sobolev spaces on a scaled torus and  Dirac's comb}
\label{SUBSEC:Sobolev}

In the proof of Theorem \ref{THM:1}, 
scaling of periodic domains plays an important role.
Hence, 
we briefly go over the basic definitions and properties 
of Sobolev spaces  
on a scaled torus $\T_L := \R /(L \Z)$,  
$L \geq 1$.\footnote{In the following, we often
identify
$\T_L$
with the interval $[-\frac L2, \frac L2] \subset \R$.} 
Given a function $f$ on $\T_L$, we define 
its Fourier coefficient  by 
\begin{equation}
 \ft{f}\big(\tfrac nL\big) 
 = \frac{1}{L} \int_{\T_L} f(x) e^{-2\pi i \frac{n}{L}x} dx, \qquad n \in \Z.
\label{F1}
 \end{equation}

\noi
Then, we have the following Fourier inversion formula:
\begin{equation} f (x) =  
\sum_{ n \in \Z } \ft{f}\big(\tfrac{n}{L}\big) e^{2\pi i \frac{n}{L} x}
\label{F2}
\end{equation}

\noi
and Plancherel's identity:
\[ \| f \|_{L^2(\T_L)} = L^\frac{1}{2}\big\| \ft f \big\|_{\l^2(\Z/L)}
= L^{\frac{1}{2}} \bigg( \sum_{n \in \Z } 
\big|\ft f\big(\tfrac{n}{L}\big)\big|^2\bigg)^\frac{1}{2}.
\]

\noi
Given $s \in \R$, 
 we define the homogeneous Sobolev space $\dot H^s(\T_L)$
and 
 the inhomogeneous Sobolev space $ H^s(\T_L)$
by the norms:
\begin{align}
\| f \|_{\dot H^s(\T_{L})} 
& := L^\frac{1}{2} \bigg( \sum_{n \in \Z \setminus \{0\} } 
\big|\tfrac{n}{L}\big|^{2s} \big|\ft f\big(\tfrac{n}{L}\big)\big|^2\bigg)^\frac{1}{2}, 
\label{F3}
\\
\| f \|_{ H^s(\T_{L})}
& := L^\frac{1}{2} \bigg( \sum_{n \in \Z } 
\big\langle\tfrac{n}{L}\big\rangle^{2s} \big|\ft f\big(\tfrac{n}{L}\big)\big|^2\bigg)^\frac{1}{2},
\label{F3a}
\end{align}

\noi
where $\jb{\, \cdot\, } = (1 + |\cdot|^2)^\frac{1}{2}$.

Recall the following definition of Dirac's comb:
\begin{align*}
\Xi_L(x) = \sum_{ n \in \Z} \dl_0(x - nL)
\end{align*}

\noi
for $L > 0$, where $\dl_0$ is the Dirac's delta function.
Given a (smooth) function $f \in L^1(\R)$, 
define the periodization $f_L$ of $f$ with period $L$
by setting
\begin{align}
f_L(x) = \Xi_L*f(x)
= \sum_{ n \in \Z} f(x - nL).
\label{F4}
\end{align}

\noi
Then, $f_L$ is a periodic function with period $L$.
Moreover, from \eqref{F1} and \eqref{F4}, we have
\begin{align}
\ft f_L\big(\tfrac nL\big)
= \tfrac{1}{L}
\ft f\big(\tfrac nL\big),
\qquad n \in \Z, 
\label{F5}
\end{align}

\noi
where $\ft f$ on the right-hand side denotes
the Fourier transform of $f$ on $\R$.
Then, it follows from  \eqref{F3} and \eqref{F5} with a Riemann sum approximation that 
\begin{align}
\| f_L \|_{\dot H^s(\T_{L})} 
& =  \bigg( \frac{1}{L}\sum_{n \in \Z \setminus\{0\}} 
\big|\tfrac{n}{L}\big|^{2s} \big|\ft f\big(\tfrac{n}{L}\big)\big|^2\bigg)^\frac{1}{2}\notag\\ 
& 
\longrightarrow
\bigg( \int_\R |\xi|^{2s} |\ft f(\xi)|^2 d\xi\bigg)^\frac{1}{2}
= \|f\|_{\dot H^s(\R)}
\label{F6}
\end{align}

\noi
as $L \to \infty$.
Similarly, we have
\begin{align}
\| f_L \|_{H^s(\T_{L})} 
\longrightarrow  \|f\|_{H^s(\R)}
\label{F7}
\end{align}

\noi
as $L \to \infty$.
Dirac's comb and related Poisson's summation formula
are typical tools to pass information from $\R$ to 
a periodic torus.
See \cite{Sobolev} for example.

\section{Proof of Theorem \ref{THM:1}}

We follow closely 
the norm inflation argument 
for \eqref{NLS0} on $\R$ by 
Christ-Colliander-Tao \cite{CCT2b}
and adapt it suitably to the periodic setting, at least when $s < -\frac 12$.
Recall that the argument in 
 \cite{CCT2b}
consists of the following three components:
\begin{itemize}
\item[(a)]
	
high-to-low energy transfer  for solutions to the following ODE:
\begin{align}
i \dt w (x) + |w|^2 w (x) = 0
\label{NLS2}
\end{align}

\noi
for $x \in \R$.
Recall that there is an explicit solution formula for \eqref{NLS2}:
\begin{align}
w(x, t) = w(x, 0) e^{i |w(x, 0)|^2  t}.
\label{NLS2a}
\end{align}

\item[(b)]
approximation property of the ODE \eqref{NLS2} by  the 
following cubic NLS with small dispersion: 
\begin{align}
i \dt v -   \dl^2  \dx^2 v +  |v|^{2}v  = 0.
\label{NLS3}
\end{align}

\noi
\item[(c)]
modified scaling argument 
to relate the small dispersion NLS \eqref{NLS3} and the cubic NLS \eqref{NLS0}.

\end{itemize}

\noi
A primary issue in adapting this argument to the periodic setting is the fact that the dilation symmetry alters
the spatial domain  in the periodic setting.
In the following, we use Dirac's comb and periodization
(of a function on $\R$)
with large periods to handle Step (b).
Then, we apply a modified scaling to 
reduce \eqref{NLS3}
 on a torus of a large period
to  \eqref{NLS0}
on the standard torus $\T$.

In the critical case $s = -\frac 12$, 
the scaling part of the argument no longer works.
We directly establish an approximation 
property of the ODE \eqref{NLS2}
by NLS \eqref{NLS0}.

\subsection{Approximation lemma}

In the following, we establish an approximation property
of  the ODE \eqref{NLS2} 
by the small dispersion NLS \eqref{NLS3}.
This approximation lemma will be useful in the supercritical case discussed
in Subsection \ref{SUBSEC:super}.
See \cite[Lemma 2.1]{CCT2b}
for an analogous statement in the non-periodic setting.

\begin{lemma} \label{LEM:approx}

Given $\phi \in C^\infty_c(\R)$, 
suppose that 
$L$ is sufficiently large such that 
$\supp \phi \subset \big[-\frac L2, \frac L2\big] \simeq\T_L$.
Let $w$ and $v = v(\dl)$ 
be the solutions to \eqref{NLS2} and \eqref{NLS3} on $\T_L$
with $w|_{t = 0} = v|_{t=0} = \phi$, respectively.
Then,  given any $\beta \in (0, 2)$, 
there exist $c>0$,  $C>0$,  $\dl_0 > 0$, and $L_0\geq1$ such that 
\begin{align*}
\| v (t)- w (t) \|_{ H^1 (\T_L)} \leq C \dl^\be
\end{align*}	

\noi
for all $|t| \leq c |\log \dl |^c$,  all $ 0 < \dl \leq \dl_0$, 
and all  periods $L \geq L_0$.

\end{lemma}

In the current periodic setting, it is important
to prove Lemma \ref{LEM:approx}
with constants $c, C,$ and $ \dl_0$,  independent of a period $L \geq 1$.
For this purpose, we recall 
the following Sobolev's embedding type estimate.  
From \eqref{F2}, \eqref{F3a},
and a Riemann sum approximation,   we have 
\begin{align}
\| f\|_{L^\infty(\T_L)} 
& \leq \sum_{n \in \Z} \big| \ft f \big( \tfrac n L \big) \big|
\leq \bigg( \sum_{n \in \Z} \frac{1}{1 + |\frac n L |^2} \cdot \frac 1 L \bigg)^\frac{1}{2}\| f \|_{H^1 (\T_L)}
\notag \\
& \leq C \| f \|_{H^1 (\T_L)}, 
\label{G1}
\end{align}

\noi
where $ C> 0$ is independent of $L \gg 1$.

\begin{proof}
Letting  $z := v - w$, we have 
\begin{align}
\begin{cases}
i \dt z - \dl^2 \dx^2 z = \dl^2 \dx^2 w - \NN(w, z)\\
z|_{t = 0} = 0, 
\end{cases}
\label{H1}
\end{align}

\noi
for $(x, t) \in \T_L \times \R$,
where
$\NN(w, z) : = |w+z|^2(w+z) - |w|^2 w$.
Assuming the following bootstrap hypothesis:
\begin{align}
\| z(t) \|_{H^1(\T_L)} \les 1, 
\label{H2}
\end{align}
 
\noi 
we need to prove
\begin{align}
\| z(t) \|_{H^1(\T_L)} \leq C \dl^\be, 
\label{H3}
\end{align}

\noi
provided 
\begin{align}
|t| \leq c |\log \dl |^c.
\label{H4}
\end{align}

\noi
for some small $c > 0$ and $0 < \dl \leq \dl_0\ll1$.

It follows from the continuity of $z(t)$ with $z(0) = 0$
that 
there exists $\eps > 0$ such that 
\eqref{H2} holds on $[-\eps, \eps]$.
In the following, we estimate $\dt \|z(t) \|_{H^1(\T_L)}$
on $[-\eps, \eps]$.

Let  $L$ be sufficiently large
such that $\supp \phi \subset \T_L $.
Then, noting that 
 $\supp w(t) = \supp \phi \subset \T_L$ for all $t \in \R$, 
it follows from  \eqref{NLS2a} that
\begin{align}
 \|  w(t) \|_{ H^k(\T_L)} 
 = 
  \|  w(t) \|_{ H^k(\R)}
\leq C  ( 1+|t|)^{k}
\label{H4a}
\end{align}

\noi
for $k \in \N$ and all sufficiently large $L \gg 1$.
Thus, 
we have
\begin{align}
\dl^2  \| \dx^2 w(t) \|_{ H^1(\T_L)} \leq 
\dl^2  \| \dx^2 w(t) \|_{ H^1(\R)}
\leq C \dl^2 ( 1+|t|)^{3}.
\label{H5}
\end{align}

\noi
By the product rule, H\"older's inequality,  and
Sobolev's embedding  \eqref{G1} with \eqref{H4a}, 
we have 
\begin{align}
\|  \NN(w, z)(t)   \|_{ H^1(\T_L)} 
\leq C(1 + |t|)^2 \| z(t) \|_{H^1(\T_L)}
+ C \| z(t) \|_{H^1(\T_L)}^3.
\label{H6}
\end{align}

\noi
From 
\eqref{H1}, \eqref{H5},
and \eqref{H6}
with the bootstrap assumption \eqref{H2}, we have
\begin{align*}
\dt \| z(t) \|_{H^1(\T_L)}
\leq C \dl^2 ( 1+|t|)^{3} + C(1 + |t|)^3 \| z(t) \|_{H^1(\T_L)}.
\end{align*}

\noi
Hence, by Gronwall's inequality with $z(0) = 0$, we obtain
\begin{align*}
\| z(t) \|_{H^1(\T_L)} \leq C \dl^2 e^{C ( 1+ |t|)^4 }, 
\end{align*}

\noi
which in turn implies \eqref{H3} as long as $t \in [-\eps, \eps]$ and \eqref{H4} holds.
Therefore,  by the continuity argument, 
we conclude that 
\eqref{H2} and hence \eqref{H3} hold as long as \eqref{H4} holds.
\end{proof}

\subsection{Supercritical case}\label{SUBSEC:super}

Let $s <  -\frac 12$.
Our goal is to construct  a sequence of solutions $u_j$ to \eqref{NLS0}, 
$j \in \N$, 
such that 
\begin{align}
 \| u_j (0)\|_{\dot H^s(\T)} <  \tfrac{1}{j} 
\qquad \text{ and } \qquad \| u_j(t_j)\|_{\dot H^s(\T)} > j
\label{X1}
\end{align}

\noi
for some $0< t_j < \frac {1}{j}$.

Given $s < - \frac 12$, let $\phi \in C^\infty_c(\R)$
with $\supp \phi \subset [-K, K]$ for some $K > 0$.
Moreover, assume that 
$\ft \phi(\xi) = O_{\xi \to 0}(|\xi|^\kappa)$ for some $\kappa > - s - \frac 12$
such that  $\phi \in  \dot H^s(\R) $. 
Let $w$ be the solution \eqref{NLS2} on $\R$
with $w|_{t = 0} = \phi$.
Note that $\supp w(t) = \supp \phi \subset [-K, K]$ for all $t \in \R$.
In view of \eqref{NLS2a}, we can choose $\phi$ such that
\begin{align}
\bigg| \int_\R w(x, 1) dx \bigg| \ges 1.
\label{X2}
\end{align}

\noi
See 
Appendix \ref{SEC:APP}
for details of the construction of such a  function $\phi$.

Given $j \in \N$, 
let 
\begin{align}
L_j = \frac{\dl_j}{\ld_j}
\label{X2a}
\end{align}

\noi
for some $0 < \ld_j \ll \dl_j \ll1 $ (to be chosen later).
Let $w_j = \Xi_{L_j} * w $ be the periodization of $w$ with period $L_j$
defined in \eqref{F4}.
By assuming $L_j \geq 2K$, 
we have  $w_j (x) = w(x) $ on 
$\T_{L_j}\simeq \big[-\frac {L_j}{2}, \frac{L_j}{2}\big] \supset [-K, K]$.
In particular, $w_j$ is a solution to \eqref{NLS2} on $\T_{L_j}$
with $w_j|_{t = 0} = \phi$.

Let $v_j$ be the solution to \eqref{NLS3} with $\dl = \dl_j$ on $\T_{L_j}$
such that  $v_j |_{t = 0} = \phi$.
By choosing $\dl_j \leq \dl_0$, 
Lemma \ref{LEM:approx} with $\be = 1$ yields
\begin{align}
\| v_j (t)- w_j (t) \|_{ H^1 (\T_{L_j})} \leq C \dl_j
\label{X3}
\end{align}	

\noi
for all $|t| \leq c |\log \dl_j |^c$.
Now, define $u_j$ by 
\begin{equation}
u_j (x, t) =  \tfrac 1{\ld_j}  v_j\big(\tfrac{\dl_j x }{\ld_j}, \tfrac{t}{\ld_j^2} \big).
\label{X3a}
\end{equation}

\noi
Then, $u_j$ is a solution to \eqref{NLS0} on $\T$.
From \eqref{F1} with \eqref{X2a}, we have  
\begin{align}
 \ft u_j(n, \ld_j ^2 t)
& =   \int_{\T} u_j(x,\ld_j ^2  t) e^{-2\pi i nx} dx
 = 
\frac{1}{\ld_j} \frac{1 }{L_j} \int_{\T_{L_j}} v_j (x,   t) e^{-2\pi i \frac{n}{L_j}x} dx \notag\\
& = \tfrac{1}{\ld_j} \ft v_j\big(\tfrac{n }{L_j}  ,t \big).
\label{X4}
\end{align}

\noi
In particular, from \eqref{X4} and \eqref{F5},  we have
\begin{align}
 \ft u_j(0, 0)
& = \tfrac{1}{\ld_j} \ft v_j(0, 0)
= \tfrac{1}{\ld_j L_j} \F_{\R} (\phi)(0) = 0.
\label{X4a}
\end{align}

\noi
From   \eqref{F3} and \eqref{X4}, we have
\begin{align}
\| u_j(t)\|_{\dot H^s(\T)}
= \ld_j^{-s -\frac{1}{2}} \dl_j^{s - \frac 12}
\big\|  v_j\big(\tfrac{t}{\ld_j^2} \big)\big\|_{\dot H^s(\T_{L_j})}.
\label{X5}
\end{align}

\noi
In view of \eqref{X4a}, 
\eqref{X5},  and 
\eqref{F6} with $v_j(0) = \phi$,
we can choose  $0 < \ld_j \ll \dl_j \ll 1$
such that 
\begin{align}
\| u_j(0)\|_{ H^s(\T)}
& \leq 
\| u_j(0)\|_{\dot H^s(\T)}
 = \ld_j^{-s -\frac{1}{2}} \dl_j^{s - \frac 12}
\|  v_j(0)\|_{\dot H^s(\T_{L_j})}\notag\\
& 
\sim \ld_j^{-s -\frac{1}{2}} \dl_j^{s - \frac 12}
\| \phi\|_{\dot H^s(\R)}
\ll \frac{1}{j}, 
\label{X6}
\end{align}

\noi
provided $s< -\frac 12$.  
In particular, we can simply choose 
$0 < \ld_j \ll \dl_j \ll 1$ such that 
\begin{align}
\ld_j^{-s -\frac{1}{2}} \dl_j^{s - \frac 12}
\sim \dl_j^\theta\ll \frac{1}{j}
\label{X6a}
\end{align}

\noi
for some small $\theta > 0$ with $s < -\frac 12 - \frac \theta 2$.

On the other hand, 
from \eqref{X2} and the continuity of $\ft w(\cdot, 1)$, 
there exists $C_0 > 0$ such that 
\begin{align*}
|\ft w(\xi, 1)| \geq \frac c2
\end{align*}

\noi
 for all  $|\xi| \leq C_0$.
Then, from  \eqref{F5}, we have
\begin{align}
\big|\ft w_j \big(\tfrac{n}{L_j} , 1\big) \big| \geq \frac{c}{2L_j}
\label{X7}
\end{align}

\noi
for all  $|n| \leq C_0L_j$.
Define $\G_j \subset \Z$ by 
\begin{align}
\G_j := \Big\{ n \in \Z : \, |n| \leq C_0 L_j, \ 
\big|\ft w_j \big(\tfrac{n}{L_j} , 1\big) -\ft v_j \big(\tfrac{n}{L_j} , 1\big) \big| \geq \frac{c}{4L_j}
\Big\}.
\label{X7a}
\end{align}

\noi
From \eqref{X7} and \eqref{X7a},  we have
\begin{align}
 \big|\ft v_j \big(\tfrac{n}{L_j} , 1\big) \big| \geq \frac{c}{4L_j}
\label{X7b}
\end{align}

\noi
on $\G_j^c$ defined by 
\begin{align*}
\G_j^c := \big\{1\leq |n|\leq C_0 L_j\big\} \setminus \G_j.
\end{align*}

\noi
Now, let us estimate the size of $\G_j$.
From \eqref{X3} and \eqref{X7a}, we have
\begin{align}
\big(\#\G_j\big)^\frac{1}{2} \cdot \frac{c}{4L_j}
\leq 
\bigg(\sum_{n \in \G_j} 
\big|\ft w_j \big(\tfrac{n}{L_j} , 1\big) - \ft v_j \big(\tfrac{n}{L_j} , 1\big) \big|^2
\bigg)^\frac{1}{2}
\leq C L_j^{- \frac 12}\dl_j
\label{X7b2}
\end{align}

\noi
for sufficiently small $\dl_j > 0$.
Thus, 
we have
\begin{align}
\#\G_j\les L_j \dl_j^2.
\label{X7c}
\end{align}

\noi
Hence, from 
\eqref{F3}, 
\eqref{X7b},  \eqref{X7c},  
 and a Riemann sum approximation, we have
\begin{align}
\|  v_j(1)\|_{ \dot H^s(\T_{L_j})}
& \ges  
L_j^\frac{1}{2}
\bigg( \sum_{n\in \G_j^c} \big|\tfrac{n}{L_j}\big|^{2s}\tfrac{1}{L_j^2}\bigg)^\frac{1}{2}
 \ge  
L_j^\frac{1}{2}
\bigg( \sum_{L_j \dl_j^2 \les |n| \leq C_0 L_j } \big|\tfrac{n}{L_j}\big|^{2s}\tfrac{1}{L_j^2}\bigg)^\frac{1}{2}\notag\\
& \sim
 \bigg( \int_{c_0 \dl_j^2 \le |\xi| \leq C_0} |\xi|^{2s}  d\xi\bigg)^\frac{1}{2}
\sim 
\dl_j^{2s+ 1}
\label{X8}
\end{align}

\noi
for all $L_j = \frac{\dl_j}{\ld_j} \gg1 $ and $\dl_j \ll 1$.

From  \eqref{X4}, 
\eqref{X2a}, and 
 \eqref{X8} with \eqref{X6a}, we conclude that 
\begin{align}
\| u_j (\ld_j^2)\|_{H^{s}(\T)}
& = \ld_j^{-1}\bigg( \sum_{n \in \Z} (1 + |n|^2)^{s} 
\big|\ft v_j\big(\tfrac{n}{L_j}, 1\big)\big|^2 \bigg)^\frac{1}{2}\notag \\
& \geq  \ld_j^{-1}
L_j^{s - \frac{1}{2}}
\cdot
L_j^\frac{1}{2}\bigg( \sum_{n \ne 0} \Big( \tfrac{1}{L_j^2} + \big|\tfrac{n}{L_j}\big|^2\Big)^{s} 
\big|\ft v_j\big(\tfrac{n}{L_j}, 1\big)\big|^2 \bigg)^\frac{1}{2}\notag \\
& \sim   \ld_j^{ -s - \frac 12} \dl_j^{s - \frac 12} \| v_j(1 ) \|_{\dot H^{s}(\T_{L_j})}
\notag\\
& \ges \dl_j^\theta \cdot \dl_j^{2s + 1} 
\gg j
\label{X10}
\end{align}

\noi
by choosing $\dl_j > 0$ sufficiently small.
Clearly, we can choose $\ld_j > 0$ sufficiently small
to guarantee that  $t = \ld_j^2 < \frac{1}{j}$.
Hence, \eqref{X1} follows from \eqref{X6} and \eqref{X10}.

\subsection{Critical case}\label{SUBSEC:crit}
Next, we consider the critical case $s  = s_\text{crit} = -\frac 12$.
In this case, the scaling argument used in the previous subsection 
does not work any longer
and 
%
%
the high-to-low energy transfer used 
in the supercritical case 
(as in \cite{CCT2b} and Subsection \ref{SUBSEC:super} presented above)
is not enough.
We need to exploit a stronger high-to-low energy transfer
for some solutions to \eqref{NLS2}.
See Lemma \ref{LEM:growth} below.

Given $N \gg 1$, define a periodic function $\phi_N$ on $\T$ by setting
\begin{align}
\ft \phi_N(n) = \ind_{N + Q_A} (n) + \ind_{2N + Q_A} (n), 
\label{phi1}
\end{align}

\noi
where $Q_A =\big[-\frac{A}{2}, \frac{A}{2}\big]$ with 
\begin{align}
\qquad A = A(N) = \frac{N}{(\log N)^\frac{1}{16}}.
\label{phi2}
\end{align}


Let $w^N$ be the  global solution to \eqref{NLS2} posed on $\T$
such that  $w^N|_{t = 0} = \phi_N$.
In view of \eqref{NLS2a}, 
we have 
\begin{align}
 w^N  (t) = \phi_N e^{i |\phi_N|^2 t}.
\label{phi3}
\end{align}

By writing 
 \eqref{NLS2} in the integral form:
\begin{align}
w^N (t) = \phi_N + i \int_0^t |w^N |^2 w^N (t') dt', 
\label{NLS2b}
\end{align}
	
\noi
it follows from the algebra property
of the Wiener algebra $\F L^1(\T)$ that
we can also construct the solution $w^N $
to \eqref{NLS2b} by a fixed point argument
on a time interval $[0, T^*_N]$ with 
\begin{align}
T^*_N \sim  \|\phi_N \|_{\F L^1(\T)}^{-2} \sim 
 \frac{(\log N)^\frac{1}{8}}{N^2}, 
\label{time1}
\end{align}

\noi	
satisfying 
\begin{align}
\| w^N (t)\|_{\F L^{1}(\T)} & \les  \| \phi_N \|_{\F L^1(\T)} 
\sim \frac{N}{(\log N)^\frac{1}{16}}
\label{phi5a}
\end{align}
	
\noi
for  all $t \in [0, T^*_N]$.
Moreover, 
by a variant of the persistence of regularity
(with the triangle inequality and Young's inequality
on the Fourier side),   we have 
\begin{align}
\| w^N (t)\|_{\F L^{\infty}(\T)} & \les  \| \phi_N \|_{\F L^\infty(\T)} 
= 1, 
\label{phi5b}
\end{align}
	
\noi
for 
 all $t \in [0, T^*_N]$.

From \eqref{phi1} and \eqref{phi2}, we have
\begin{align}
\| \phi_N \|_{H^{-\frac{1}{2}}(\T)} \sim(\log N)^{-\frac{1}{32}}, 
\label{phi6}
\end{align}

\noi
which tends to 0 as $N \to \infty$.
The following lemma
exhibits a crucial norm inflation for the ODE \eqref{NLS2} 
in the critical regularity.

\begin{lemma}\label{LEM:growth}
Given $N \gg1$, define $T_N > 0$ by 
\begin{align}
T_N = \frac{1}{N^2 (\log N)^\frac{1}{8}}.
\label{time2}
\end{align}
	
\noi
Then, we have 
\begin{align*}
\| \P_{< N} w^N (T_N)\|_{H^{-\frac{1}{2}}(\T)} 
& \ges 
 (\log N)^{\frac 14}.
\end{align*}

\noi
Here, $\P_{< N}$ denotes the projection onto the frequencies
$\{ |n| < N\}$.
\end{lemma}

Lemma \ref{LEM:growth} follows
from expressing \eqref{phi3} in 
the power series:
\begin{align*}
 w^N  (t) = 
 \sum_{k = 0}^\infty
\Xi_k(t) 
:= 
 \sum_{k = 0}^\infty
 \frac{(it)^k}{k!}
  |\phi_N|^{2k}  \phi_N
\end{align*}

\noi
and estimating each term, either from below or above.
The main contribution comes from $\Xi_1$.
We present the proof of Lemma \ref{LEM:growth} in Appendix \ref{SEC:APP2}.

Let $u^N$ be  the solution to \eqref{NLS0} with $u^N|_{t = 0} = \phi_N$.
We denote 
by  $\u^N$ 
the interaction representation of $u^N$
defined by $\u^N (t)=  e^{it\partial^2_{x}} u^N(t)$.
Then, from the Duhamel formulation of $u^N$,\footnote{Henceforth, we drop 
the factor $2\pi$ when it plays no role.}
 we have
\begin{align}
\u^N (t) 
& =  \phi_N + i \int_0^t \sum_{n\in \Z} e^{inx} \notag \\
& \hphantom{XXXXX}
\times\sum_{\Gamma(n)} e^{- i\Phi(\bar n)t'} 
\ft{\u^N}(n_1, t') \cj{\ft{\u^N}(n_2, t')} \ft{\u^N}(n_3, t')dt', 
\label{dul}
\end{align}

\noi
where $\Phi(\bar n) = n^2-n_1^2+n_2^2-n_3^2$ and 
\[
\G(n) = \{(n_1,n_2,n_3) \in \Z^3:\,  n = n_1 - n_2 + n_3\}.
\]

\noi
Then, arguing as above for the solution $w^N$ to \eqref{NLS2b}, 
it follows from \eqref{dul}, the algebra property
of $\F L^1(\T)$, and a variant of the persistence of regularity
that 
\begin{align}
\|\u^N (t)\|_{\F L^1(\T)} & \les  \|\phi_N \|_{\F L^1(\T)}, 
\label{YW1}\\
\| \u^N (t)\|_{\F L^\infty(\T)} & \les  \| \phi_N \|_{\F L^\infty(\T)},  
\label{YW2}
\end{align}
for all $t\in [0, T^*_N]$, where $T^*_N$ is as in \eqref{time1}.
In the following, we use $w^N$ to approximate the interaction 
representation $\u^N$.

\begin{lemma}\label{LEM:YW}
Given $N \in \N$, let $w^N$ and $\u^N$ be as above.
Then,  for any $T\in (0, T^*_N]$, we have
\begin{align*}
\|w^N- & \u^N\|_{L^\infty_T \F L^\infty}\notag\\
& \les  \bigg\| \sum_{n = n_1-n_2+n_3}
\int_0^t (1 -  e^{- i\Phi(\bar n)t'} ) dt'
\, 
\ft \phi_N(n_1)  
\cj{\ft \phi_N(n_2) } \ft \phi_N(n_3)  \bigg\|_{L^\infty_T \F L^\infty}  \notag\\
& \hphantom{XXXXX}
+ \sum_{k=2}^4 T^k \| \phi_N \|^{2k}_{\F L^1(\T)} \|\phi_N\|_{\F L^\infty (\T)}, 
\end{align*}

\noi
where the implicit constant is independent of $N \in \N$
and 
$L^\infty_T \F L^\infty: = L^\infty([0, T];  \F L^\infty(\T))$.
\end{lemma}

We first present the proof of Theorem \ref{THM:1},
assuming Lemma \ref{LEM:YW}.
Given  $j \in \N$, fix  $N = N(j) \gg 1$
such that 
\begin{align}
\frac{1}{N^2 (\log N)^\frac{1}{8}}\ll
\frac{1}{(\log N)^{\frac{1}{32}}} \ll \frac 1j 
\qquad\text{and}\qquad
(\log N)^\frac{1}{4} \gg j.
\label{pf1}
\end{align}

\noi
Let   $T_N$  as in \eqref{time2}.
%
 By the mean value theorem, we have 
\begin{align}
\label{YW3}
\sup_{t \in [0, T_N]}\bigg|\int_0^t (1 -  e^{- i\Phi(\bar n)t'} ) dt'\bigg|
\leq T_N^2 |\Phi(\bar n)| \les T_N^2 N^2,
\end{align}

\noi
provided that 
$|n|, |n_j|\les N$, $j = 1, 2, 3$.
Then, by Lemma \ref{LEM:YW} and \eqref{YW3}
with \eqref{phi5a} and  \eqref{phi5b}, 
 we have 
\begin{align*}
\|w^N- & \u^N\|_{L^\infty_{T_N} \F L^\infty}\notag\\
& \les   T_N^2  N^2 \| \phi_N \|^{2}_{\F L^1(\T)} \|\phi_N\|_{\F L^\infty (\T)} 
+ \sum_{k=2}^4 T_N^k \| \phi_N \|^{2k}_{\F L^1(\T)} \|\phi_N\|_{\F L^\infty (\T)}\\
& \les  (\log N)^{-\frac38}.
\end{align*}

\noi
In particular, we have
\begin{align}
\|\P_{<N} \big( & w^N(T_N)  - \u^N(T_N)\big)\|_{H^{-\frac12}(\T)} \notag \\
& \les (\log N)^{\frac12} \|w^N(T_N)-\u^N(T_N) \|_{ \F L^\infty(\T)} \les (\log N)^{\frac18}.
\label{YW4}
\end{align}

\noi
Then, from Lemma \ref{LEM:growth}
and \eqref{YW4}, we conclude that 
\begin{align}
\| u^N (T_N)\|_{H^{-\frac{1}{2}}(\T)} 
&  = \| \u^N (T_N)\|_{H^{-\frac{1}{2}}(\T)} \notag\\
 & \geq   \|\P_{<N} \u^N (T_N)\|_{H^{-\frac{1}{2}}(\T)} 
\ges 
 (\log N)^{\frac 14}.
\label{YW5}
\end{align}

\noi
Therefore, 
it follows from \eqref{phi6} and  \eqref{YW5} with \eqref{time2} and \eqref{pf1}
that 
the desired estimates in \eqref{X1}  
hold with $t_j := T_N \ll \frac{1}{j}$.

It remains to present the proof of Lemma \ref{LEM:YW}.

\begin{proof}[Proof of  Lemma \ref{LEM:YW}]
Note that Lemma \ref{LEM:YW} immediately follows once we prove
 \begin{align}
  \bigg\| \u^N &  - \phi_N -i  \int_0^t e^{i\partial_x^2 t '} \big(| e^{-i\partial_x^2 t '} \phi_N|^2 
  e^{-i\partial_x^2 t '}  \phi_N \big) dt' \bigg\|_{L^\infty_{T} \F L^\infty (\T)} \notag \\ 
&  \les  \sum_{k=2}^4 T^k \| \phi_N \|^{2k}_{\F L^1(\T)} \|\phi_N\|_{\F L^\infty (\T)}
 \label{error1}
 \end{align}
 
 \noi
 and 
 \begin{align}
  \bigg\| w^N   - \phi_N -i  \int_0^t &  |\phi_N|^2\phi_N dt' \bigg\|_{L^\infty_{T} \F L^\infty (\T)} \notag\\
&   \les  \sum_{k=2}^4 T^k \| \phi_N \|^{2k}_{\F L^1(\T)} \|\phi_N\|_{\F L^\infty (\T)}
 \label{error2}
 \end{align}

\noi
for $0 \leq T \leq T^*_N$.
We only prove \eqref{error1}
since \eqref{error2} follows in a similar manner.
In the following, we drop the superscripts and subscripts $N$
for simplicity.
Moreover, we set $\phi_n = \ft \phi(n)$, $ \u_n (t) = \ft \u(n, t)$,  and so on.

 In view of the \eqref{dul}, we have 
 \begin{align}
  \u_n(t) & =  \phi_n + i \int_0^t \sum_{\Gamma(n)} e^{- i\Phi(\bar n)t'}  \u_{n_1} \cj  \u_{n_2}  \u_{n_3} dt'
   = :   \phi_n + \NN(\u)_n .
\label{YW6}
 \end{align}

\noi
By substituting $\u = \phi + \NN(\u)$ in \eqref{YW6}, 
we obtain 
  \begin{align*}
  \u_n(t)  
& =  \phi_n + i \int_0^t  \sum_{\Gamma(n)} e^{- i\Phi(\bar n)t'} 
( \phi_{n_1} +\NN(\u)_{n_1}) \\
& \hphantom{XXXXXXXX}\times (\cj{ \phi_{n_2} +\NN(\u)_{n_2}}) 
( \phi_{n_3} +\NN(\u)_{n_3}) dt'  \notag\\
 & =  \phi_n + i \int_0^t   \sum_{\Gamma(n)} 
 e^{- i\Phi(\bar n)t'} \phi_{n_1} \cj{\phi}_{n_2}  \phi_{n_3} dt' + \text{Error}_n, 
\end{align*}

\noi
where the terms in $\text{Error}_n$ consist
of higher order terms of the form:
\begin{align}
i  \int_0^t   \sum_{\Gamma(n)} 
 e^{- i\Phi(\bar n)t'} v_{n_1} \cj{v}_{n_2}  v_{n_3} dt'
\label{YW7}
 \end{align}

\noi
with $v = \phi$ or $\NN(\u)$, not all three factors equal to $\phi$.
Then, it suffices to show that 
 \begin{align}
\sup_{t \in [0, T]} \sup_n |\text{Error}_n| \les \sum_{k=2}^4 T^k \| \phi \|^{2k}_{\F L^1(\T)} \|\phi\|_{\F L^\infty (\T)}.
\label{YW8}
 \end{align}

\noi
Note that 
\eqref{YW8}
follows from \eqref{YW7}, 
\eqref{YW6}, 
and Young's inequality
with \eqref{YW1} and \eqref{YW2}.
The proof of  \eqref{error2} follows in a similar manner
with \eqref{phi5a} and  \eqref{phi5b}.
\end{proof}

\noi
This completes the proof of Theorem \ref{THM:1}
in the critical case.

\section{On norm inflation for the cubic fractional NLS}

In this section, we present the proof of Theorem \ref{THM:2}
for the fractional NLS \eqref{HNLS1}.

\subsection{Supercritical case}
In this case, the following cubic fractional NLS with small dispersion: 
\begin{align}
i \dt v +   \dl^{2\al}(-  \dx^2)^\al v +  |v|^{2}v  = 0
\label{HNLS2}
\end{align}

\noi
plays an important role.
We state a variant of Lemma \ref{LEM:approx}.

\begin{lemma} \label{LEM:approx2}

Given $\phi \in C^\infty_c(\R)$, 
suppose that 
$L$ is sufficiently large such that 
$\supp \phi \subset \T_L$.
Let $w$ and $v = v(\dl)$ 
be the solutions to \eqref{NLS2} and \eqref{HNLS2} on $\T_L$
with $w|_{t = 0} = v|_{t=0} = \phi$, respectively.
Then, given any $\be \in (0, 2\al)$, there exist $c>0$,  $C>0$,  $\dl_0 > 0$, and $L_0 \geq 1$ such that 
\begin{align*}
\| v (t)- w (t) \|_{ H^1 (\T_L)} \leq C \dl^\be
\end{align*}	

\noi
for all $|t| \leq c |\log \dl |^c$,  all $ 0 < \dl \leq \dl_0$, 
and all  periods $L \geq L_0$.

\end{lemma}

The proof of Lemma \ref{LEM:approx2} is basically the same as that of Lemma \ref{LEM:approx}
and hence we omit details.

We divide the supercritical case $s < s_\text{crit}^\al = \frac 12 - \al $
into the following three cases:
\[ \text{(i) } s \leq - \tfrac 12, \qquad
\text{(ii) } - \tfrac 12 < s < 0, 
\qquad \text{and}
\qquad 
\text{(iii) } s > 0.\]

\noi
Recall that there is no inflation when $s = 0$ due to the $L^2$-conservation.
When $s \leq - \frac 12$, 
the basic structure of the proof is the same as that of Theorem \ref{THM:1}.
When $s > -\frac 12$, however, 
we need to employ a more robust high-to-low energy transfer mechanism
for the ODE \eqref{NLS2}
as in Subsection \ref{SUBSEC:crit}.

\smallskip

\noi
$\bullet$ {\bf Case (i)} $s \leq - \frac 12$.
\quad 
In this case, we follow the argument presented in 
Subsection \ref{SUBSEC:super}.
Let $w$  be the solution to \eqref{NLS2} 
on $\R$ with $w|_{t = 0} = \phi$
considered  
in Subsection \ref{SUBSEC:super}.
See Appendix \ref{SEC:APP} 
for the construction of such $\phi$.
In particular, $w$ satisfies \eqref{X2}.
Given $j \in \N$, let $0 < \ld_j \ll \dl_j \ll1$.
With $L_j = \frac{\dl_j}{\ld_j}\gg 1$ as in \eqref{X2a}, 
let $w_j$ and $v_j$ be the solutions to \eqref{NLS2} 
and \eqref{HNLS2} with $\dl = \dl_j$ 
on $\T_{L_j}$
such that  $w_j|_{t = 0} = v_j |_{ t= 0} = \phi$.
By choosing $\dl_j \leq \dl_0$, 
Lemma \ref{LEM:approx2} yields
\begin{align}
\| v_j (t)- w_j (t) \|_{ H^1 (\T_{L_j})} \leq C \dl_j^\al
\label{Z2}
\end{align}	

\noi
for all $|t| \leq c |\log \dl_j |^c$.
Now, define $u_j$ by 
\begin{equation}
u_j (x, t) =  \tfrac 1{\ld_j^\al}  v_j\big(\tfrac{\dl_j x }{\ld_j}, \tfrac{t}{\ld_j^{2\al}} \big).
\label{Z3}
\end{equation}

\noi
Then, $u_j$ is a solution to \eqref{HNLS1} on $\T$.
From \eqref{F1}, we have 
\begin{align}
 \ft u_j(n, \ld_j ^{2\al} t)
 = \tfrac{1}{\ld_j^\al} \ft v_j\big(\tfrac{n }{L_j}  ,t \big).
\label{Z3b}
 \end{align}

\noi
As before,  we have
\begin{align}
 \ft u_j(0, 0)
& = \tfrac{1}{\ld_j^\al} \ft v_j(0, 0)
= \tfrac{1}{\ld_j^\al L_j} \F_{\R}( \phi)(0) = 0.
\label{Z3a}
\end{align}

\noi
From \eqref{F3} and \eqref{Z3b}, 
 we have 
\begin{align}
\| u_j(t)\|_{\dot H^s(\T)}
= \ld_j^{-s +\frac{1}{2} - \al} \dl_j^{s - \frac 12}
\big\|  v_j\big(\tfrac{t}{\ld_j^{2\al}} \big)\big\|_{\dot H^s(\T_{L_j})}.
\label{Z4}
\end{align}

\noi
In particular, 
in view of 
\eqref{Z3a}, 
\eqref{Z4},  and 
\eqref{F6} with $v_j(0) = \phi$,
we can choose 
$0 < \ld_j \ll \dl_j \ll 1$
such that
\begin{align}
\| u_j(0)\|_{ H^s(\T)}
& \leq 
\| u_j(0)\|_{\dot H^s(\T)} 
\sim \ld_j^{-s +\frac{1}{2}- \al } \dl_j^{s - \frac 12}
\| \phi\|_{\dot H^s(\R)}
\ll \frac{1}{j}, 
\label{Z5}
\end{align}

\noi
since $s< s_\text{crit}^\al = \frac 12 - \al$.  
In particular, we can  choose 
$0 < \ld_j \ll \dl_j \ll 1$ such that 
\begin{align}
\ld_j^{-s +\frac{1}{2} - \al } \dl_j^{s - \frac 12}
\sim 
\begin{cases}
\dl_j^\theta\ll \frac{1}{j}, & \text{if } s < -\frac 12, \\
\big\{\log (C_1 \dl_j^{-2\al})\big\}^{-\frac{1}{4}}\ll \frac{1}{j}, & \text{if } s =  -\frac 12, \\
\end{cases}
\label{Z5a}
\end{align}

\noi
for some small $\theta > 0$ with $s < -\frac 12 - \frac \theta {2\al}$
in the first case, 
where $C_1 > 0$ is to be specified later.

Let $\G_j$ be as in \eqref{X7a}.
By repeating the computation in \eqref{X7b2} with \eqref{Z2}, 
we obtain
\begin{align}
\#\G_j\les L_j \dl_j^{2\al}.
\label{Z5b}
\end{align}

\noi
Using  \eqref{Z5b}, 
we proceed 
 as in \eqref{X8},
 where the domain of the integration is 
 replaced by $\{ c_0 \dl_j^{2\al} \leq |\xi| \leq C_0\}$.
 Then, setting   $C_1 := c_0^{-1} C_0$, 
we have
\begin{align}
\|  v_j(1)\|_{\dot H^s(\T_{L_j})}
\ges 
\begin{cases}
\dl_j^{\al(2s+ 1)}\gg 1, & \text{if }s < -\frac 12,\\
\big\{\log (C_1 \dl_j^{-2\al})\big\}^\frac{1}{2} \gg 1,  & \text{if } s = -\frac 12.
\end{cases}
\label{Z6}
\end{align}

\noi
Finally,  proceeding as in \eqref{X10}
with 
\eqref{Z3b}, 
\eqref{X2a}, 
 \eqref{Z5a},  and \eqref{Z6}, we obtain 
\begin{align}
\| u_j (\ld_j^{2\al})\|_{H^{s}(\T)}
& = \ld_j^{-\al}\bigg( \sum_{n \in \Z} (1 + |n|^2)^{s} 
\big|\ft v_j\big(\tfrac{n}{L_j}, 1\big)\big|^2 \bigg)^\frac{1}{2}\notag \\
& \geq  \ld_j^{-\al}
L_j^{s - \frac{1}{2}}
\cdot
L_j^\frac{1}{2}\bigg( \sum_{n \ne 0} \Big( \tfrac{1}{L_j^2} + \big|\tfrac{n}{L_j}\big|^2\Big)^{s} 
\big|\ft v_j\big(\tfrac{n}{L_j}, 1\big)\big|^2 \bigg)^\frac{1}{2}\notag \\
& \sim   \ld_j^{ -s + \frac 12 - \al } \dl_j^{s - \frac 12} \| v(1 ) \|_{\dot H^{s}(\T_{L_j})}
\notag\\
& \ges 
\begin{cases}
\dl_j^\theta\cdot \dl_j^{\al(2s+ 1)} \gg j,  & \text{if }s < -\frac 12,\\
\big\{\log (C_1 \dl_j^{-2\al})\big\}^\frac{1}{4} \gg j, 
& \text{if } s = -\frac 12, 
\end{cases}
\label{Z7}
\end{align}

\noi
by choosing $\dl_j > 0$ sufficiently small.
We can also choose 
$\ld_j > 0$ sufficiently small such that 
  $t_j = \ld_j^{2\al} < \frac{1}{j}$.
Hence, Theorem \ref{THM:2} follows from \eqref{Z5} and \eqref{Z7}
in this case.

\medskip

\noi
$\bullet$ {\bf Case (ii)} $- \frac 12 < s < 0$.
\quad 
Note that the high-to-low energy transfer 
manifested in \eqref{X2} is not sufficient
to show a norm inflation
when $s > -\frac 12$.
Hence, we first need to construct another solution to the ODE \eqref{NLS2},
exhibiting a more robust 
high-to-low energy transfer.

Given $N \gg 1$, define a function $\psi_N$ on $\R$ by setting
\begin{align}
\ft \psi_N(\xi) =R\big\{ \ind_{N + Q_A} (\xi) + \ind_{2N + Q_A} (\xi)\big\},  
\label{psi1}
\end{align}

\noi
where $R = R(N)$ and $A  = A(N)$ are given by 
\begin{align*}
R = R(N) = \frac{1}{N^s \log N}
\qquad \text{and}
\qquad A = A(N) = \log N.
\end{align*}

\noi
From \eqref{psi1}, we have
\begin{align}
\| \psi_N \|_{H^s(\R)} \sim(\log N)^{-\frac{1}{2}}.
\label{psi3}
\end{align}

We first state 
a norm inflation for the ODE \eqref{NLS2}
posed on $\R$
when $ - \frac 12 < s < 0$.

\begin{lemma}\label{LEM:growth2}
Let $- \frac 12 < s  < 0$.
Given $N \gg1$, let $w_N$ be 
the global solution to \eqref{NLS2} posed on $\R$
with  $w_N|_{t = 0} = \psi_N$.
Then, we have 
\begin{align*}
\|  w_N (T_N)\|_{H^s(\R)} \ges N^{-s}(\log N)^{-\frac 32 + s}, 
\end{align*}

\noi
where $T_N$ is defined by 
\begin{align}
T_N = \frac{N^{2s}}{\log N}.
\label{ptime2}
\end{align}

\end{lemma}

The proof of Lemma \ref{LEM:growth2} 
is analogous to that of 
Lemma \ref{LEM:growth}
and is presented 
 in Appendix \ref{SEC:APP2}.

Given large $j \in \N$, 
choose $N = N(j) \gg 1$ such that 
\begin{align}
\frac{N^{2s}}{\log N}
 \ll \frac 1{(\log N)^{\frac{1}{2}}} \ll \frac{1}{j}
 \qquad \text{and}
 \qquad
  N^{-s}(\log N)^{-\frac 32 + s}
 \gg  j.
\label{psi5}
\end{align}

\noi
Given $s < 0$, fix $\kappa \in \N$ such that $s + \kappa \geq 0$
and let $\psi^\kappa_N = \dx^{-\kappa} \psi_N$.
More precisely, 
$\psi^\kappa_N$ is defined by 
$\ft{\psi^\kappa_N}(\xi)  = (2\pi i \xi)^{-\kappa} \ft {\psi_N}(\xi)$.
Given small $\eta > 0$ (to be chosen later), 
let $\phi^\kappa \in C^\infty_c(\R)$
such that 
\begin{align}
\|  \phi^\kappa - \psi_N^\kappa \|_{H^{\kappa+1}(\R)} <  \eta.
\label{Psi4}
\end{align}

\noi
Define $\phi$ by setting $\phi = \dx^\kappa \phi^\kappa$.
Then, from \eqref{Psi4} have 
\begin{align}
\|  \phi - \psi_N \|_{H^{1}(\R)} 
\leq 
\|  \phi^\kappa - \psi_N^\kappa \|_{H^{\kappa+1}(\R)} <  \eta.
\label{Psi5}
\end{align}

\noi
Recalling that $s + \kappa \geq  0$, 
it follows from  \eqref{Psi4} and \eqref{psi1} that
\begin{align}
\| \P_{\leq 1} \phi\|_{\dot H^s(\R)}
& = \| \P_{\leq 1} \phi^\kappa\|_{\dot H^{s+\kappa}(\R)}\notag\\
& \leq
\|  \phi^\kappa - \psi_N^\kappa \|_{H^{\kappa+1}(\R)} 
+ \| \P_{\leq 1} \psi_N^\kappa\|_{\dot H^{s+\kappa}(\R)}
<  \eta.
\label{Psi6}
\end{align}

\noi
Hence, from \eqref{F6}, \eqref{psi3}, \eqref{Psi5}, and \eqref{Psi6} with \eqref{psi5},  
we have
\begin{align}
\| \phi\|_{\dot H^s(\T_L)}
\sim
\| \phi\|_{\dot H^s(\R)}
& \les  (\log N)^{-\frac{1}{2}} + \eta \ll \tfrac{1}{j}, 
\label{Psi6a}
\end{align}

\noi
for all sufficiently large $L\gg1$, 
provided that  $\eta \ll \frac{1}{j}$.

Let $w$ be the solution to \eqref{NLS2} on $\R$
with $w |_{t = 0} = \phi$.
Then, by the continuity  of the solution map 
to \eqref{NLS2} (in the $H^1$-topology)
and \eqref{Psi5}, 
we can choose $\eta = \eta(N) > 0$ sufficiently small
to guarantee that 
\begin{align}
\|w(t) - w_N(t) \|_{H^1(\R)} \les 1
\label{psi8}
\end{align}

\noi
for all $t \in [0, T_N]$.
Then, from Lemma \ref{LEM:growth2}
with \eqref{psi5} and \eqref{psi8}, we have
\begin{align}
\| w (T_N)\|_{H^s(\R)} \gg j.
\label{psi9}
\end{align}

Now, recalling that $s < s_\text{crit} = \frac 12 - \al$, 
choose $0 < \ld_j \ll \dl_j \ll1$
such that 
\begin{align}
\ld_j^{-s +\frac{1}{2} - \al} \dl_j^{s - \frac 12} \sim 1
\label{psi10}
\end{align}

\noi
and $\supp \phi \subset \T_{L_j}$, where
$L_j = \frac{\dl_j}{\ld_j} \gg 1$ is as in \eqref{X2a}.
Moreover,  from \eqref{F7}
and \eqref{psi9}, we have  
\begin{align}
\| w (T_N)\|_{H^s(\T_{L_j})} \gg j
\label{psi11}
\end{align}

\noi
by choosing $L_j$ sufficiently large.
Let $v_j$ be the solution to \eqref{HNLS2} 
with $\dl = \dl_j$ on $\T_{L_j}$
such that $v_j|_{t = 0} = \phi$.
Then, 
it follows from 
Lemma \ref{LEM:approx2} applied  to $w$ and $v_j$ on $[0, T_N]$
with \eqref{ptime2}, \eqref{psi5}, 
and \eqref{psi11} that
\begin{align*}
\| v_j (T_N)\|_{H^s(\T_{L_j})} \gg j.
\end{align*}

\noi
This in particular implies
\begin{align}
\| v_j (T_N)\|_{\dot H^s(\T_{L_j})} \gg j
\qquad \text{or}
\qquad 
L_j^\frac{1}{2}|\ft v_j (0, T_N)| \gg j.
\label{psi12}
\end{align}

Now, define $u_j$ by \eqref{Z3}.
On the one hand,  it follows from 
\eqref{Z3a}, \eqref{Z4},  \eqref{Psi6a}, and \eqref{psi10}
that 
\begin{align}
\| u_j(0)\|_{ H^s(\T)}
 \leq 
\| u_j(0)\|_{\dot H^s(\T)}  
 =  \ld_j^{-s +\frac{1}{2}- \al } \dl_j^{s - \frac 12}\| \phi\|_{\dot H^s(\T_{L_j})}
\ll \frac{1}{j}.
\label{psi13}
\end{align}

\noi
On the other hand, 
from \eqref{Z3b}, \eqref{Z4}, \eqref{psi10}, 
and \eqref{psi12}, we obtain
\begin{align}
\| u_j (\ld_j^{2\al} T_N)\|_{H^{s}(\T)}
& \sim  \| u_j (\ld_j^{2\al} T_N)\|_{\dot H^{s}(\T)}
+ |\ft u_j(0, \ld_j^{2\al} T_N)|\notag\\
& = \ld_j^{-s +\frac{1}{2}- \al } \dl_j^{s - \frac 12}\| v_j(T_N) \|_{\dot H^s(\T_{L_j})}
+  \tfrac{1}{\ld_j^\al} |\ft v_j(0 ,T_N)|\notag\\
& \sim  \| v_j(T_N) \|_{\dot H^s(\T_{L_j})}
+   L_j^{\frac{1}{2}-s}| \ft v_j(0 ,T_N)|\notag\\
& \gg j.
\label{psi14}
\end{align}

\noi
Note that we have
$t_j = \ld_j^{2\al} T_N \ll \frac 1j$ in view of 
 \eqref{ptime2} and \eqref{psi5}.
Hence, Theorem \ref{THM:2} follows from \eqref{psi13} and \eqref{psi14}
in this case.

\medskip

\noi
$\bullet$ {\bf Case (iii)} 
Lastly, we consider the case $0 < s < \frac 12 - \al$.
Fix a non-constant mean-zero function $\phi$ with a compact support on $\R$
and let $w$ be the solution to \eqref{NLS2} on $\R$
with $w |_{t = 0} = \phi$.
From  a direct calculation,  we have
\begin{align*}
\|  w(t)\|_{\dot H^k(\T_{L})}= 
\|  w(t)\|_{\dot H^k(\R)}\sim (1+ t)^k.
\end{align*}

\noi
for all $k \in \N \cup\{0\}$ and all sufficiently large $L\gg 1$.
Then, by interpolation, we have
\begin{align}
\|  w(t)\|_{\dot H^\s(\T_{L})}\sim (1+t)^\s
\label{D1a}
\end{align}

\noi
for all $\s\geq 0$.

Given $j \in \N$, 
choose $0 < \ld_j \ll \dl_j \ll 1$
such that 
\eqref{D1a} holds for $\s = s$ and $ L = L_j
 : = \frac{\dl_j}{\ld_j}\gg 1$.
Let $v_j$ be the solution to \eqref{HNLS2} with $\dl = \dl_j$ on $\T_{L_j}$
such that  $v_j |_{t = 0} = \phi$.
Then, by 
Lemma \ref{LEM:approx2} and \eqref{D1a}, there exists $c_0 > 0$ such that 
\begin{align}
\|  v_j(t)\|_{\dot H^s(\T_{L_j})}\sim (1+t)^s.
\label{D2}
\end{align}

\noi
\noi
for $ |t| \leq c_0 |\log \dl_j|^{c_0}$, provided 
$ \dl_j \ll 1$ and $L_j \gg 1$.
In the following, we further impose that 
\begin{align}
\ld_j^{-s +\frac{1}{2} - \al} \dl_j^{s - \frac 12}
\sim  |\log \dl_j|^{- \frac 12c_0s}.
\label{D3}
\end{align}

\noi
Then, 
it follows from 
\eqref{Z3a}, \eqref{Z4},  \eqref{D3}, 
and 
\eqref{F6} that 
\begin{align}
\| u_j(0)\|_{ H^s(\T)}
& \les 
\| u_j(0)\|_{\dot H^s(\T)} 
\sim 
|\log \dl_j|^{- \frac 12c_0s}
\| \phi\|_{\dot H^s(\R)}
\ll \frac{1}{j}
\label{D4}
\end{align}

\noi
for sufficiently small $\dl_j >0$
and $L_j \gg 1$.
Letting $\wt t_j = c_0 |\log \dl_j|^{c_0}$, 
it follows 
from \eqref{Z4}, \eqref{D2}, \eqref{D3}
that 
\begin{align}
\| u_j(\ld_j^{2\al} \wt t_j)\|_{ H^s(\T)}
\geq 
\| u_j(\ld_j^{2\al} \wt t_j)\|_{\dot H^s(\T)}
\sim 
  |\log \dl_j|^{ \frac 12c_0s}\gg j
\label{D5}
\end{align}

\noi
\noi
for sufficiently small $\dl_j >0$.
From \eqref{D3}, we have
 $\ld_j \les \dl_j^\beta$ for some $\beta > 0$.
Hence, by choosing $\dl_j >0 $ sufficiently small, we have
\begin{align}
t_j : = \ld_j^{2\al} \wt t_j
 \les  \dl_j^{2 \al \be}   \cdot 
 c_0 |\log \dl_j|^{c_0}\ll \frac 1j.
\label{D6}
\end{align}

\noi
Therefore, Theorem \ref{THM:2} follows from \eqref{D4}, \eqref{D5},  and \eqref{D6}
in this case.

\subsection{Critical case}  \label{SUBSEC:crit2}
Next, we consider the critical case $s  = s_\text{crit}^\al = \frac 12 - \al \leq -\frac 12 $.
In view of Theorem \ref{THM:1}, 
we assume that $s < -\frac 12$.
The argument 
follows closely 
the presentation  in Subsection \ref{SUBSEC:crit}
with Lemma \ref{LEM:growth3} below
replacing Lemma \ref{LEM:growth}.

Given $N \gg 1$, define a periodic function $\phi_N$ on $\T$ by setting
\begin{align}
\ft \phi_N(n) = R\big\{ \ind_{N + Q_A} (n) + \ind_{2N + Q_A} (n)\big\}, 
\label{C1}
\end{align}

\noi
where $R = R(N)$ and $A  = A(N)$ are given by 
\begin{align}
R = R(N) = N^{-\frac 12 - s} \qquad \text{and}\qquad
A = A(N) = N^{1 - \theta}, 
\label{C1a}
\end{align}

\noi
with sufficiently small $\theta > 0$
such that 
\begin{align*}
s < -\tfrac 12 - 3 \theta.
\end{align*}

\noi
From \eqref{C1} and \eqref{C1a}, we have 
\begin{align}
  \| \phi_N \|_{H^{\frac 12 - \al}(\T)} 
\sim N^{- \frac \theta 2 }.
\label{C2a}
\end{align}

Let $w^N$ be 
the  global solution to \eqref{NLS2} posed on $\T$
such that  $w^N|_{t = 0} = \phi_N$.
Proceeding as in Subsection \ref{SUBSEC:crit}, 
we have
\begin{align}
\| w^N (t)\|_{\F L^{1}(\T)} & \les  \| \phi_N \|_{\F L^1(\T)} 
\sim N^{\frac 12 - s - \theta}
\label{C3a},\\
\| w^N (t)\|_{\F L^{\infty}(\T)} & \les  \| \phi_N \|_{\F L^\infty(\T)} 
\sim N^{-\frac 12 - s}
\label{C3b}
\end{align}
	
\noi
for 
 all $t \in [0, T^*_N]$, 
 where $T^*_N$ is given by 
 \begin{align}
T^*_N \sim  \|\phi_N \|_{\F L^1(\T)}^{-2} \sim 
N^{2 s - 1 + 2\theta}.
\label{Ctime1}
\end{align}

The following lemma is a variant of Lemma \ref{LEM:growth} for $s < -\frac 12$.
See Appendix \ref{SEC:APP2}
for the proof.

\begin{lemma}\label{LEM:growth3}
Let $ s< -\frac 12 $. 
Given $N \gg1$, define $T_N > 0$ by 
\begin{align}
T_N =N^{2 s - 1 - \theta}.
\label{Ctime2}
\end{align}
	
\noi
Then, we have 
\begin{align*}
\| \P_{<N} w^N (T_N)\|_{H^{\frac{1}{2}-\al}(\T)} \ges N^{-\frac 12 - s - 3\theta}. 
\end{align*}

\end{lemma}

Let $u^N$ be  the solution to \eqref{HNLS1} with $u^N|_{t = 0} = \phi_N$.
Then, 
let  $\u^N$ be 
the interaction representation of $u^N$
defined by $\u^N (t)=  e^{it(-\partial_x^2)^{\al}} u^N(t)$.
From the Duhamel formulation of $u^N$,
 we have
\begin{align}
\u^N (t) 
& =  \phi_N + i \int_0^t \sum_{n\in \Z} e^{inx} \notag \\
& \hphantom{XXXXX}
\times\sum_{\Gamma(n)} e^{- i\Phi_\al(\bar n)t'} 
\ft{\u^N}(n_1, t') \cj{\ft{\u^N}(n_2, t')} \ft{\u^N}(n_3, t')dt', 
\label{ZW1}
\end{align}

\noi
where $\Phi_\al(\bar n) = |n|^{2\al}-|n_1|^{2\al}+|n_2|^{2\al}-|n_3|^{2\al}$. 
Then, as in Subsection \ref{SUBSEC:crit}, we have
\eqref{YW1} and \eqref{YW2}
for all $t\in [0, T^*_N]$, where $T^*_N$ is as in \eqref{Ctime1}.
Moreover, the approximation lemma (Lemma \ref{LEM:YW})
also holds with $w^N$, $\u^N$, and $T^*_N$
in our current context.

%
%
%
%
%


Given  $j \in \N$, fix  $N = N(j) \gg 1$
such that 
\begin{align}
N^{2 s - 1 - \theta}
\ll
N^{- \frac \theta 2 } \ll \tfrac 1j 
\qquad\text{and}\qquad
N^{-\frac 12 - s - 3\theta}\gg j.
\label{C4a}
\end{align}

\noi
Let   $T_N = N^{-2\al - \theta}$  as in \eqref{Ctime2}.
Then,  by the mean value theorem, we have 
\begin{align}
\label{ZW3}
\sup_{t \in [0, T_N]}\bigg|\int_0^t (1 -  e^{- i\Phi(\bar n)t'} ) dt'\bigg|
\leq T_N^2 |\Phi_\al(\bar n)| \les T_N^2 N^{2\al},
\end{align}

\noi
provided that 
$|n|, |n_j|\les N$, $j = 1, 2, 3$.
Then, by Lemma \ref{LEM:YW} and \eqref{ZW3}
with \eqref{C3a} and  \eqref{C3b}, 
 we have 
\begin{align}
\|w^N- & \u^N\|_{L^\infty_{T_N} \F L^\infty}\notag\\
& \les   T_N^2  N^{2\al} \| \phi_N \|^{2}_{\F L^1(\T)} \|\phi_N\|_{\F L^\infty (\T)} 
+ \sum_{k=2}^4 T_N^k \| \phi_N \|^{2k}_{\F L^1(\T)} \|\phi_N\|_{\F L^\infty (\T)}\notag \\
& \les  N^{-\frac{1}{2} - s - 4\theta}.
\label{ZW3a}
\end{align}

\noi
In particular, we have
\begin{align}
\|\P_{<N} \big(  w^N(T_N)  - \u^N(T_N)\big)\|_{H^{s}(\T)} 
&   \les  \|w^N(T_N)-\u^N(T_N) \|_{ \F L^\infty(\T)} \notag \\
&  \les N^{-\frac{1}{2} - s - 4\theta}.
\label{ZW4}
\end{align}

\noi
Then, from Lemma \ref{LEM:growth3}
and \eqref{ZW4}, we conclude that 
\begin{align}
\| u^N (T_N)\|_{H^{s}(\T)}  & = \| \u^N (T_N)\|_{H^{s}(\T)} \notag\\
&  \geq  \| \P_{<N} \u^N (T_N)\|_{H^{s}(\T)} 
\ges 
N^{-\frac{1}{2} - s - 3\theta}.
\label{ZW5}
\end{align}

\noi
Therefore, 
Theorem \ref{THM:2} in the critical case
 follows from \eqref{C2a} and  \eqref{ZW5} with \eqref{Ctime2} and \eqref{C4a}.

\begin{remark}\rm
It follows from the proof of Lemma \ref{LEM:growth3} (see Lemma \ref{LEM:Hs3} below) that
\begin{align}
\| \P_{<N}w^N (T_N)\|_{H^{s}(\T)} \ges 
T_N \| \phi_N \|^{2}_{\F L^1(\T)} \|\phi_N\|_{\F L^\infty (\T)} 
\sim N^{-\frac 12 - s - 3\theta}
\label{ZW3b}
\end{align} 

\noi
and $T_N N^{2\al} \sim N^{-\theta}$.
Comparing \eqref{ZW3a} and \eqref{ZW3b}, we see that it is essential
to have $T_N N^{2\al} \les 1$ for our argument to work.
\end{remark}

\begin{remark}\label{REM:fail1} \rm
The proof of Theorem \ref{THM:2}
in the critical case with $s < -\frac 12$ is based on 
Lemma \ref{LEM:growth3}, exploiting a high-to-low energy transfer analogous
to Lemmas \ref{LEM:growth} and \ref{LEM:growth2}
along with the approximation lemma (Lemma \ref{LEM:YW}). 
When $-\frac 12 < s < 0$, one may hope to exploit 
a similar high-to-low energy transfer. 
There are,  however, no possible values of $R$ and $A$
for  an initial condition $\phi_N$ of the form \eqref{C1}, 
guaranteeing norm inflation at time $T_N \les N^{-2\al}$.
See Remark \ref{REM:fail2}.
\end{remark}

\appendix

\section{High-to-low energy transfer for the ODE: Part 1}
\label{SEC:APP}

Let $s \leq - \frac 12$.
In the following, we discuss the construction of the function  
 $\phi \in C^\infty_c(\R)\cap \dot H^s(\R) $ described in 
Subsection \ref{SUBSEC:super}, 
satisfying  
(i) $\ft \phi(\xi) = O_{\xi \to 0}(|\xi|^\kappa)$ for some $\kappa > - s - \frac 12$
and (ii) 
there exist $t_0 > 0$ and $c > 0$ such that 
\begin{align}
\bigg| \int_\R w(x, t_0) dx \bigg| \geq c, 
\label{A1}
\end{align}

\noi
where  $w$ is the solution \eqref{NLS2} on $\R$
with $w|_{t = 0} = \phi$.
In view of  \eqref{NLS2a}, 
we can rewrite \eqref{A1} as 
\begin{align}
\bigg| \int_\R \phi(x) e^{  i |\phi(x)|^2 t_0}  dx \bigg| \geq c. 
\label{A1a}
\end{align}

\noi
The main addition from \cite{CCT2b} is the compactness of the support of $\phi$,
which is crucial for our application in the periodic setting.

The following lemma states that the smoothness assumption on $\phi$ 
is inessential.

\begin{lemma} \label{LEM:const}
Let $s \leq - \frac {1}{2}$.
Suppose that there exists a function $\psi \in C_c^0(\R) \cap\dot H^s (\R)$
 such that 
\textup{(i)} $\ft \psi(\xi) = O_{\xi \to 0}(|\xi|^\kappa)$ for some $\kappa > - s - \frac 12$
and \textup{(ii)} 
there exist $t_0 > 0$ and $c > 0$ such that 
\begin{align}
\bigg| \int_\R \psi(x) e^{  i |\psi(x)|^2 t_0}  dx \bigg| \geq c. 
\label{A2}
\end{align}

\noi
Then, there exists  $\phi \in C^\infty_c(\R)\cap \dot H^s(\R) $
satisfying both the conditions \textup{(i)} and \textup{(ii)}
with $c$ replaced by $\frac c2$.

\end{lemma}

\begin{proof}
It follows from Paley-Wiener Theorem \cite[Theorem IX.12]{RS}
that 
$\ft \psi$ is an analytic function on $\R$.
Hence, the vanishing of $\ft \psi$ at 0  must occur at an integral order.
This allows us to assume that $\kappa \in \N$.
Then, from the Taylor expansion $\ft \psi(\xi) = \sum_{j = 0}^\infty \frac{\dd^j \ft \psi(0)}{j!} \xi^j$, 
we  have $\dd^j \ft \psi(0) = 0$ for $j = 0, 1, \dots, \kappa-1.$
In other words, we  have 
\begin{align}
\int_\R x^j \psi(x) dx = 0
\label{A3}
\end{align}

\noi
for $j = 0, 1, \dots, \kappa-1.$

Let $\eta \in C^\infty_c(\R)$ be a smooth bump function
with $\supp \eta \subset [-1, 1]$
and $\int \eta dx = 1$.
Given $\eps > 0$ (to be chosen later), 
let $\phi_\eps = \eta_\eps * \psi$,
where  $\eta_\eps (x) = \eps^{-1} \eta(\eps^{-1}x)$.
Then, we  claim that 
\begin{align}
\int_\R x^j \phi_\eps(x) dx = 0
\label{A4}
\end{align}

\noi
for $j = 0, 1, \dots, \kappa-1.$
Given $j \in \{0, 1, \dots, \kappa - 1\}$, write $x^j$ as 
\begin{align*}
x^j = \big((x-y) + y\big)^j = \sum_{k = 0}^j c_k(y) (x-y)^k.
\end{align*}

\noi
Then, by Fubini's theorem, we have
\begin{align*}
\int_\R x^j \phi_\eps(x) dx 
= 
\sum_{k = 0}^j 
\int\bigg(  \int  (x-y)^k \psi(x - y) dx \bigg)
c_k(y) \eta_\eps(y) dy  = 0.
\end{align*}

\noi
Hence, \eqref{A4} holds.
Noting $\phi_\eps$ has a compact support, 
we conclude from 
the analyticity of $\ft \phi_\eps$
that the condition (i) also holds for $\phi_\eps$.

For $0< \eps \leq 1$, we have  $\supp \phi_\eps \subset \supp \psi + B_0(1)$, 
where $ B_0(1)$ denotes the ball of radius 1 centered at 0. 
In particular, $\phi_\eps$ converges to $\psi$ in $L^\infty(\R)$.
Then, by the triangle inequality and the mean value theorem, we have 
\begin{align}
\bigg| \int_\R & \psi(x) e^{  i |\psi(x)|^2 t_0}  dx 
-  \int_\R \phi_\eps(x) e^{ i |\phi_\eps(x)|^2 t_0}  dx \bigg| \notag\\
& \leq 
   \int_\R |\psi(x)| |e^{  i |\psi(x)|^2 t_0}
- e^{  i |\phi_\eps(x)|^2 t_0}|
  dx 
  + 
 \int_\R |\psi(x) - \phi_\eps(x) | dx \notag\\
& \leq c(\psi, t_0)  \int_{\supp \psi + B_0(1)} |\psi(x) - \phi_\eps(x) | dx 
\leq \frac {c}{2}
\label{A5}
\end{align}

\noi
for all sufficiently small $\eps > 0$.
Therefore, 
from \eqref{A2} and \eqref{A5}, we conclude that 
the condition (ii) with $c$ replaced by $\frac c2$ also holds for $\phi_\eps$ 
with sufficiently small $\eps > 0$.
\end{proof}

For the first few values of $\kappa \in \N$, 
we  concretely
construct functions $\psi_\kappa$,   satisfying the conditions (i) and (ii) 
in Lemma \ref{LEM:const}.
Then, Lemma \ref{LEM:const}
yields a smoothed version $\phi_\kappa$ of $\psi_\kappa$,  
serving as a good initial 
condition $\phi$  
in the supercritical case of the proof of Theorem \ref{THM:1} and \ref{THM:2} 
when $ - \kappa -\frac 12 < s \leq - \frac 12$.

Define $\psi_1$ by 
\[\psi_1 (x) =  \sqrt{\pi} \ind_{[1, 3]} (x) - 2 \sqrt{\pi} \ind_{[4, 5]}(x). \]

\noi
Note that $\ft \psi_1(0) = 0$.
Thus, by the analyticity of $\ft \psi_1$, 
we have 
 $\ft \psi_1(\xi) = O_{\xi \to 0}(|\xi|)$.
The condition (ii) is clearly satisfied with $t_0 = 1$.

Now, let  $\psi_2(x) = \psi_1(x) + \psi_1(-x)$.
It is easy to see that $\psi_2$
 satisfies \eqref{A3} for $j = 0, 1$.
 Thus, by analyticity of $\ft \psi_2$, it satisfies the condition (i) with $\kappa = 2$.
The condition (ii) is clearly satisfied with $t_0 = 1$.
	
Next, define $\psi_4$ with a parameter $a$ by 
\[ \psi_4 (x; a) =  \sqrt{\pi} \ind_{[1, 2]} (x) - 2 \sqrt{\pi} \ind_{[4, 5]}(x)
+  \sqrt{\pi} \ind_{[a, a+1]} (x), \quad x \geq 0  \]
	
\noi
and $\psi_4(x; a) = \psi_4(-x; a)$ for $x < 0$, where $a > 5$.
By definition, $\psi_4$ satisfies \eqref{A3}
for $j = 0, 1,$ and $ 3$.
Letting $F(a) = \int_\R x^2 \psi_4(x; a) dx$, we see that 
$F(5) < 0 < F(10)$.  Hence, by the intermediate value theorem,
there exists $a_* \in (5, 10)$
such that $F(a_*) = 0$.
Now,  set $\psi_4(x) = \psi_4(x; a_*)$.
Then, it is easy to see that $\psi_4$
satisfies the conditions (i) and (ii)\footnote{We can write
 $\psi_4 = \sqrt{\pi} \ind_{E_1} - 2 \sqrt\pi\ind_{E_2}$, 
 where 
 $ E_1 \cap E_2 = \emptyset$. 
It is clear that functions of this form satisfy \eqref{A2}
at $t_0 = 1$
with $c  = \sqrt {\pi} |E_1|+2 \sqrt {\pi} |E_2| >0$.} with $\kappa = 4$.

In general, one may continue to construct  $\psi_\kappa$ in this fashion
(and construct smooth $\phi_\kappa$ by applying Lemma \ref{LEM:const})
but  combinatorics gets cumbersome.
In the following, 
we instead discuss an alternative (simpler but less direct) construction of 
 $\phi_\kappa \in C^\infty_c(\R)\cap \dot H^s(\R) $, $\kappa \in \N$,
 satisfying the conditions (i) and (ii).

Let $w$ be a smooth solution to \eqref{NLS2}
with a smooth initial condition $w|_{t = 0} = \phi$. 
Then, from \eqref{NLS2}, we have
\begin{align*}
\dt^2 \int_\R w \,  dx = - \int_\R |w|^4 w\, dx.
\end{align*}

\noi
In particular, by taking a   mean-zero real-valued smooth  initial condition $\phi \not \equiv 0$
with a compact support such that $\int \phi^5 \, dx < 0$,
we have
\begin{align*}
\dt^2 \int_\R w \,  dx\bigg|_{t = 0}  = - \int_\R \phi^5 \, dx > 0.
\end{align*}

\noi
Then, by continuity in time of $w$, we have 
\begin{align}
 \dt^2\Re  \int_\R w (t)  dx
=   - \Re \int_\R |w|^4 w(t) dx
> 0
\label{A7}
\end{align}

\noi
on some time interval $[0, t_0]$.
On the other hand, from \eqref{NLS2}, we have 
\begin{align}
\dt \Re \int_\R w \,  dx \bigg|_{t = 0} = \Re\bigg( i  \int_\R \phi^3\, dx\bigg) = 0.
\label{A8}
\end{align}

\noi
Hence, it follows from \eqref{A7} and \eqref{A8}
that 
\begin{align*}
\dt \Re \int_\R w (t)  dx > 0
\end{align*}

\noi
for $t \in (0, t_0]$.
Recalling that $w|_{t = 0} = \phi$ has mean 0, 
we obtain
\begin{align*}
 \Re \int_\R w (t)  dx > 0
\end{align*}

\noi
for $t \in (0, t_0]$.
This implies \eqref{A1}.

Given $\kappa \in \N$, 
let $f$ be a real-valued smooth function on $\R$ with a compact support such that 
$\int (\dx^\kappa f)^5 \, dx < 0$.
Then, 
it is easy to see that $\phi_\kappa  := \dx^\kappa f \in C^\infty_c(\R)$
satisfies the conditions (i) and (ii).

\section{High-to-low energy transfer for the ODE: Part 2}
\label{SEC:APP2}

In this appendix, we 
construct solutions to the ODE \eqref{NLS2}
with more robust high-to-low energy transfer
than those constructed in Appendix \ref{SEC:APP}.

Let $\M = \T$ or $\R$ 
and $\ft \M$ denote
 the Pontryagin dual of $\M$
given by $\ft \M = \Z$ if $\M = \T$ and
$\ft \M = \R$ if $\M = \R$.
Let 
 $s < 0$.
Given $N \gg 1$, define a  function $\phi_N$ on $\M$ 
by
\begin{align}
\ft \phi_N(\xi) = R\big\{ \ind_{N + Q_A} (\xi) + \ind_{2N + Q_A} (\xi)\big\}, 
\quad \xi \in \ft \M, 
\label{Q1}
\end{align}
	
\noi
for suitably chosen
 $A = A(N)\gg 1$ and $R = R(N)\geq 1$, 
where 
$Q_A = \big[-\frac{A}{2}, \frac{A}{2}\big]$.
Note that we have
\begin{align}
\| \phi_N\|_{H^s(\M)} \sim R A^{\frac 12} N^s.
\label{Q2}
\end{align}

Let $w = w(N)$ be the solution to \eqref{NLS2} 
with $w|_{t = 0} = \phi_N$. 
Then, Lemmas \ref{LEM:growth}, 
 \ref{LEM:growth2}, and  \ref{LEM:growth3}
 follow as a corollary to the 
 following proposition.

\begin{proposition}\label{PROP:main}
Given $N \gg1$, 
set $s < 0$,  
$R = R(N)$, $A = A(N)$, and $T_N > 0$ 
by one of the followings: 
\begin{align}
\textup{(i)} \
& s = -\tfrac 12,   \quad R  = 1, 
\quad A = \frac{N}{(\log N)^{\frac 1{16}}}
\quad \text{and}\quad
T_N = \frac{1}{N^2 (\log N)^\frac{1}{8}}, 
\label{Q3}\\
\textup{(ii)} \ 
 & 
 s < 0, \quad
R = \frac{1}{N^s \log N}, 
\quad 
A  = \log N, 
\quad \text{and}
\quad 
T_N = \frac{N^{2s} }{\log N}, 
\label{Q4}\\
\intertext{\text{or \ } }
\textup{(iii)} 
\
& s < -\tfrac 12, \quad 
R  = N^{-\frac 12 - s},  
\quad
A  = N^{1 - \theta}, 
\quad \text{and}\quad
T = N^{2s - 1 - \theta}, 
\label{Q5}
\end{align}

\noi
where $\theta > 0$ is  sufficiently small 
such that 
\begin{align}
s < -\tfrac 12 - 3 \theta.
\label{Q5a}
\end{align}

\noi
Then, we have 
\begin{align}
\| w (T_N)\|_{H^{s}(\M)} 
& \geq \|\P_{<N} w (T_N)\|_{H^{s}(\M)} \notag\\
& \ges 
\begin{cases}
(\log N)^\frac{1}{4}, & \text{if \textup{(i)} holds,}\\
N^{-s} (\log N)^{-2}g(N), & \text{if \textup{(ii)} holds,}\\
N^{-\frac 12 - s - 3 \theta}, 
& \text{if \textup{(iii)} holds,}
\end{cases}
\label{Q6}
\end{align}

\noi
where $g(N)$ is given by 
\begin{align*}
g(N) = 
\begin{cases}
1, & \text{if } s < -\frac 12, \\
(\log\log N)^\frac{1}{2} & \text{if } s= -\frac 12, \\
(\log N)^{ \frac 12 + s} & \text{if } -\frac 12 < s < 0.
\end{cases}
\end{align*}

\end{proposition}

Note that, when (ii) holds, 
we only use Proposition \ref{PROP:main} for $-\frac 12 < s< 0$.
We nonetheless include the proof for all $s < 0$ in the following.

\begin{remark}
\rm 
From \eqref{Q2}, we have
\begin{align}
\|\phi_N\|_{H^{s}(\M)} 
\sim 
\begin{cases}
(\log N)^{-\frac{1}{32}}, & \text{if \textup{(i)} holds,}\\
(\log N)^{-\frac{1}{2}}, & \text{if \textup{(ii)} holds,}\\
N^{-\frac \theta 2}, 
& \text{if \textup{(iii)} holds,}
\end{cases}
\label{Q7}
\end{align}

\noi
all tending to 0 as $N \to \infty$.

\end{remark}

\noi

From the explicit formula \eqref{NLS2a}
and the power series expansion, 
we have 
\begin{align}
 w  (t) = 
  \phi_N e^{i |\phi_N|^2 t}
  = 
 \sum_{k = 0}^\infty \Xi_k(t) , 
\label{B3}
\end{align}

\noi
where $\Xi_k$ is defined by 
\begin{align}
\Xi_k(t) : = 
 \frac{(i t)^k}{k!}
  |\phi_N|^{2k}  \phi_N.
\label{B4}
\end{align}

\noi
We prove Proposition \ref{PROP:main}
by estimating each $\Xi_k$ either from above or below.
We first state elementary lemmas.

\begin{lemma}\label{LEM:Hs2}

Let $s < 0$.
Then, there exists $C >0$ such that 
\begin{align}
\| \Xi_k (t)\|_{H^s(\M)}
\leq \frac{C^k t^{k}}{k!} 
(RA)^{2k} R\cdot f(A)
\label{Hs21}
\end{align}
	
\noi
for any $k \in \N$, 
where $f(A)$ is given by 
\begin{align}
f(A) = 
\begin{cases}
1, & \text{if } s < -\frac 1 2, \\
(\log A)^\frac 12, & \text{if } s = -\frac 1 2, \\
A^{\frac{1}{2}+s} & \text{if } s > -\frac 1 2.
\end{cases}
\label{Hs21a}
\end{align}

\end{lemma}
	
\begin{proof}

From \eqref{Q1}, 
we see that $\supp \ft \phi_N$
consists of two disjoint intervals of length $ A$.
Since 
$\Xi_k(t)$ is basically 
a $(2k+1)$-fold product of $\phi_N$
and its complex conjugate, 
it follows that the
spatial support of $\F[ \Xi_k(t) ]$
consists of (at most) $2^{2k+1}$
intervals of length $ A$.
Thus, we have 
\begin{align*}
\big|\supp \F[\Xi_k(t)]\big| \leq C^k A.
\end{align*}

\noi
for some $C>0$.
By the monotonicity of $\jb{\xi}^s$ for $s < 0$, 
we have
\begin{align}
\| \jb{\xi}^s\|_{L^2_\xi(\supp \F[\Xi_k (t)])}
& \leq \| \jb{\xi}^s\|_{L^2_\xi\big([-\frac 12 C^k A,\frac 12 C^k A] \big)}
 \les 
C^k f(A).
\label{Hs22}
\end{align}

\noi
Then, by H\"older's inequality, 
\eqref{Hs22}, and Young's inequality with \eqref{Q1}, 
we have 
\begin{align*}
\| \Xi_k(t) \|_{H^s(\M)}
& \leq \| \jb{\xi}^s\|_{L^2_\xi(\supp \F[\Xi_k (t)])}
\| \Xi_k(t) \|_{\F L^\infty}  \notag\\
& \leq 
f(A)\cdot
\frac{C^k t^{k} }{k!} \| \phi_N\|_{\F L^1}^{2k}
\| \phi_N\|_{\F L^\infty} \notag\\
& \leq \frac{C^k t^{k} }{k!}
(RA)^{2k} R\cdot  f(A).
\end{align*}

\noi
This proves \eqref{Hs21}.
\end{proof}

In the next lemma, 
we  show that Lemma \ref{LEM:Hs2}
is indeed sharp when $k = 1$, 
by 
 exploiting
a high-to-low energy transfer mechanism in $\Xi_1$.

\begin{lemma}\label{LEM:Hs3}
 Let $s < 0$ and $A \ll N$.
Then,  we have
\begin{align}
\|  \Xi_1 (t)\|_{H^s}
\geq \| \P_{<N} \Xi_1 (t)\|_{H^s}
\ges t  R^3 A^{2} \cdot f(A),
\label{Hs31}
\end{align}

\noi
where $f(A)$ is as  in \eqref{Hs21a}.
\end{lemma}

\begin{proof}

First,  recall the following simple lemma on the 
convolution of characteristic functions of intervals:
\begin{align}
\ind_{a + Q_A}* \ind_{b + Q_A} (\xi)\ges A \cdot \ind_{a+b+Q_A}(\xi)
\label{B5}
\end{align}
		
\noi
for all $a, b, \xi\in \ft \M$ and $A \geq 1 $.
Then, 
from \eqref{B4}, \eqref{Q1}, and \eqref{B5}, we have 
\begin{align*}
\big| \F\big[\Xi_1(t)\big](\xi)\big|
& =  t 
 \big| \ft \phi_N*
\ft{\cj{ \phi_N}}* \ft\phi_N(\xi)\big|
\ges t R^3 A^{2} \cdot \ind_{Q_A}(\xi).
\end{align*}

\noi
Then, \eqref{Hs31} follows
once we note 
that 
$\| \jb{\xi}^s\|_{L^2_\xi( Q_A)}
\sim f(A)$
and $A\ll N$.
\end{proof}

We now present the proof of Proposition \ref{PROP:main}.

\begin{proof}[Proof of Proposition \ref{PROP:main}]

From \eqref{Q3}, \eqref{Q4}, and \eqref{Q5}, we have
\begin{align*}
T_N R^2 A^2 = 
\begin{cases}
(\log N)^{-\frac{1}{4}} \ll 1, & \text{if \textup{(i)} holds,}\\
(\log N)^{-1} \ll 1, 
& \text{if \textup{(ii)} holds,}\\
N^{- 3 \theta} \ll1 , 
& \text{if \textup{(iii)} holds.}
\end{cases}
\end{align*}

\noi
Then, 
from  \eqref{B3}
with \eqref{Q7} and Lemma \ref{LEM:Hs2}, we have
\begin{align}
\|  w(T_N) - \Xi_1(T_N) \|_{H^s}
& \leq 
\|  \Xi_0(T_N) \|_{H^s}
+  \bigg\|  \sum_{k = 2}^\infty \Xi_k (T_N)\bigg\|_{H^s}\notag\\
& \les
\begin{cases}
1, & \text{if \textup{(i)} holds,}\\
N^{-s} (\log N)^{-3}g(N),  
& \text{if \textup{(ii)} holds,}\\
N^{-\frac{\theta}{2} }+ N^{-\frac 12 - s - 6 \theta}, 
& \text{if \textup{(iii)} holds.}
\end{cases}
\label{B6}
 \end{align}

\noi
On the other hand, from Lemma \ref{LEM:Hs3}, we have
\begin{align}
\|   \Xi_1 (T_N)\|_{H^s} 
& \geq \|  \P_{<N}\Xi_1 (T_N)\|_{H^s}\notag\\
& \ges 
\begin{cases}
(\log N)^\frac{1}{4}, & \text{if \textup{(i)} holds,}\\
N^{-s} (\log N)^{-2}g(N), & \text{if \textup{(ii)} holds,}\\
N^{-\frac 12 - s - 3 \theta}, 
& \text{if \textup{(iii)} holds.}
\end{cases}
\label{B7}
\end{align}

\noi
Therefore, 
the desired estimate \eqref{Q6} follows from \eqref{B6} and \eqref{B7}
with \eqref{Q5a}.
\end{proof}

\begin{remark}\label{REM:fail2} \rm
Let us briefly discuss the situation 
for Theorem \ref{THM:2} in the critical case
when 
$s = s_\text{crit}^\al = \frac 12 - \al \in (-\frac 12 , 0)$.
In order to prove an analogue
of Proposition \ref{PROP:main} for $-\frac 12 < s < 0$, 
the following must hold:
\begin{align*}
\text{(a)} \hphantom{X} & RA^{\frac 12} N^s =: D \ll 1, \notag\\
\text{(b)} \hphantom{X} & TR^2A^2 \ll 1, \notag\\
\text{(c)} \hphantom{X} &  TR^3A^{\frac{5}{2}+s} \gg 1, 
\end{align*}
	
\noi
with $A = A(N) \ll N$ and $D = D(N) \to 0$ as $N \to \infty$.
Moreover, in carrying out the argument in Subsection \ref{SUBSEC:crit2}, we
also need to have
\begin{align*}
\text{(d)} \hphantom{X} &  T \les N^{-2\al}. 
\end{align*}

%

Let $A = EN$ 
and $T = F N^{2s-1}$
for some $E \ll 1$ and $F \les 1$.
Then, from (a),  (b), and (d), we have 
\begin{align*}
TR^2A^2 = EF\cdot R^2 AN^{2s} = D^2EF.
\end{align*}
	
\noi
Then, from (c), we obtain
\begin{align*}
1 \ll TR^3 A^{\frac 52 + s}
= E^s\cdot TR^2A^2\cdot R A^\frac {1}{2} N^s
= D^3 E^{1+s} F.
\end{align*}

\noi
Hence, we must have 
\[D^{-3} E^{-1-s} \ll F \les 1.\]

\noi
This is clearly a contradiction since $D \to 0$, $ E \ll1$ and $-\frac 12 < s < 0$.
Therefore, even if norm inflation at the critical regularity
holds true in this case, one needs to develop a new method to prove it.

\end{remark}

\begin{ackno}\rm
T.O. was supported by the European Research Council (grant no.~637995 ``ProbDynDispEq'').
T.O. would like to thank Nobu Kishimoto, Thomas Kappeler,  Pieter Blue, and Oana Pocovnicu
for interesting discussions.
The authors are grateful to the anonymous referee for the comments.
\end{ackno}

\end{document}